\documentclass[a4paper]{article}

\usepackage{amsmath, amsthm, amsfonts, amssymb,accents}
\usepackage{graphicx,color}
\usepackage{srcltx}
\usepackage{dsfont}

\theoremstyle{plain}
\newtheorem{theorem}{Theorem}[section]

\newtheorem{lemma}[theorem]{Lemma}

\theoremstyle{remark}

\numberwithin{equation}{section}

\def\tht{\theta}
\def\Om{\Omega}

\def\e{\varepsilon}
\def\g{\gamma}
\def\G{\Gamma}
\def\l{\lambda}
\def\p{\partial}
\def\D{\Delta}

\def\E{\mbox{\rm e}}
\def\a{\alpha}
\def\b{\beta}

\def\L{\Lambda}

\def\Odr{\mathcal{O}}
\def\H{W_2}

\def\Ho{\mathring{W}_2}
\def\ga{\mathring{\g}}

\def\di{\,d}

\def\I{\mathrm{I}}

\def\Op{\mathcal{H}}

\def\fh{\mathfrak{h}}

\def\cc#1{\mathcal{#1}}

\def\ph{\widehat{\phi}}

\DeclareMathOperator{\RE}{Re}
 \DeclareMathOperator{\spec}{\sigma}
\DeclareMathOperator{\discspec}{\sigma_{d}}
\DeclareMathOperator{\essspec}{\sigma_{e}}

\newcounter{assumption}
\begin{document}
\allowdisplaybreaks

\title{\textbf{Planar waveguide with ``twisted'' boundary conditions: discrete spectrum}}
\author{Denis Borisov\,$^a$, Giuseppe Cardone$^b$}
\date{\small
\begin{center}
\begin{quote}
\begin{enumerate}
{\it
\item[$a)$]
Bashkir State Pedagogical University,
October St.~3a, 450000 Ufa,
Russian Federation; \texttt{borisovdi@yandex.ru}
\item[$b)$]
University of Sannio,
Department of Engineering, Corso Garibaldi,
107, 82100 Benevento, Italy; \texttt{giuseppe.cardone@unisannio.it}
}
\end{enumerate}
\end{quote}
\end{center}
%
%
}
\maketitle

\begin{abstract}
We consider a planar waveguide with combined Dirichlet and Neumann conditions imposed in a ``twisted'' way. We study the discrete spectrum and describe it dependence on the configuration of the boundary conditions. In particular, we show that in certain cases the model can have discrete eigenvalues emerging from the threshold of the essential spectrum. We give a criterium for their existence and construct them as convergent holomorphic series.
\end{abstract}
%

\section{Introduction}

During last three decades the models of quantum waveguides attract a lot of attention of both physicists and mathematicians. The waveguides are usually modeled by infinite planar strips and multidimensional cylinders or layers. In such domains elliptic operators with Dirichlet condition are considered. Much efforts were concentrated on understanding spectral properties of quantum waveguides with various perturbations. As the examples of possible perturbations we mention local deformation of the boundary \cite{BEGK}, \cite{BGRS}, \cite{BR}, \cite{EV3}, perturbation by a potential \cite{DE} or by a second order differential operator \cite{GRU2}, adding a magnetic field \cite{BorEkhKov}, \cite{EK}, bending \cite{CDFK}, \cite{BZ}, \cite{DE}, \cite{EV3}, \cite{DES}, \cite{ES}, \cite{EFK}, \cite{GRU1} or twisting \cite{BKRS}, \cite{EKK}, \cite{EK2}, \cite{GRU1}, \cite{KS} the waveguides, see also the references in the cited papers. Waveguides with general abstract perturbation of the operator were considered in \cite{GRR04}.  One more type of the perturbation is changing the type of the boundary condition on the part of the boundary. Quite a popular model of this kind was a waveguide with a finite Neumann part on the boundary \cite{BorExGad02}, \cite{BEJPA04}, \cite{BEJMP06}, \cite{BGRS}, \cite{EV1}, \cite{GRR04}. A more general model is two waveguides having a common boundary where a gap is cut out \cite{Bor-MSb06}, \cite{BJPA07}, \cite{EV2}, \cite{E}, \cite{HTWKHKP}, \cite{K}, \cite{P}. In the cited paper the Neumann segment or a gap on the boundary were referred to as ``window(s)''. The waveguide with a magnetic field and a window was considered in \cite{BorEkhKov}. In all cited papers the authors  studied the dependence of the discrete eigenvalues on the window(s). The conditions for the existence and absence of the discrete spectrum were established. It was found that  generally the presence of the windows generates discrete eigenvalues below the essential spectrum. This phenomenon was studied in details, see \cite{BGRS}, \cite{DK}, \cite{EV1}, \cite{EV2}  and other papers. The most full and completed results are contained in \cite{BorExGad02}, \cite{Bor-MSb06}, \cite{BJPA07}, \cite{GRR04}.

A continuation of aforementioned papers on the waveguides with windows are the works where the measure of the windows is infinite. In \cite{BCJPA09}, \cite{BBCCRM11}, \cite{BBCAHP11}, \cite{BBCJMS111} the waveguides with an infinite number of windows were considered. The windows were modeled by a Neumann boundary condition on an infinite periodic set of small closely spaced segments. This is an example of perturbation from homogenization theory being studied quite well in the case of bounded domains. The main results of \cite{BCJPA09}, \cite{BBCCRM11}, \cite{BBCAHP11}, \cite{BBCJMS111} are the uniform resolvent convergence to homogenized operator and description of the asymptotic behavior of the spectrum. The next example of the waveguide with infinite Neumann part of boundary is contained in \cite{DK2}, \cite{KK}. Here a planar bent waveguide was considered with Dirichlet condition on one side and with Neumann condition on the other. It was found that the subject to the direction of bending such model possesses or does not discrete eigenvalue below the essential spectrum.  The third and the most important for us example is in the paper \cite{DK}. Here the waveguide was modeled by the Laplacian with combined Dirichlet and Neumann condition imposed in a ``twisted'' way, see figure~\ref{pic1}.  For all value of $\ell$ the domain of such operator was explicitly described. It was also shown that there exists a number $\ell_1$ such that for $\ell\leqslant \ell_1$ the discrete spectrum of the system is empty while for $\ell>\ell_1$ it is not.

\begin{figure}
\begin{center}\includegraphics[scale=0.6]{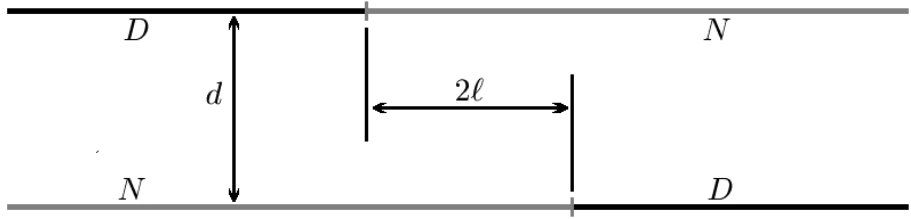}

\caption{Waveguide with combined boundary conditions} \label{pic1}\end{center}
\end{figure}

In this paper we study the above described model from  \cite{DK}, i.e., we consider the waveguide with combined Dirichlet and Neumann condition as shown on figure~\ref{pic1}. We study the dependence of the discrete spectrum on the parameter $\ell$. The results can be splitted into two groups. The first of them is formulated in Theorem~\ref{th2.2} below and consists of the statements describing general structure of the discrete spectrum. Namely, we prove that there exists an infinite set of critical values of $\ell$ such that, passing each of these, critical value creates one more discrete eigenvalue. We obtain two-sided estimates for all the discrete eigenvalues as well as two-sided estimates for the critical values of $\ell$. We also show that the eigenfunctions of the system have certain parity under the symmetry transformation w.r.t. to the center of the waveguide.

The second group of the results is formulated in Theorem~\ref{th2.3} and describes the structure of the spectrum as $\ell$ is close to a critical value $\ell_*$. We prove a criterium for a given value $\ell$ to be critical. Then we show that an additional aforementioned discrete eigenvalue emerges from the threshold of the essential spectrum. We calculate the complete asymptotic expansion for the emerging eigenvalue and provide a recurrent procedure allowing to determine all the coefficient of this expansion. The expansion is constructed in terms of a small parameter $\e:=\ell-\ell_*$. Moreover, we prove that the asymptotics converges to the eigenvalue and the sum is a holomorphic function w.r.t. $\e$. In other words, our asymptotic procedure allows one to determine the eigenvalue \emph{exactly}. The described results are similar by its nature to those in \cite{BorExGad02}, \cite{Bor-MSb06}, where the quantum waveguide with one finite window was considered. At the same, there is one important difference. In our case we succeeded to prove the holomorphy w.r.t. $\e$ of the emerging eigenvalue and we also suggest a recurrent procedure for determining the coefficients of its Taylor series. This result is in a big contrast with those in \cite{BorExGad02}, \cite{Bor-MSb06}, since in the cited papers only the leading terms of the asymptotic expansion for the emerging eigenvalues were constructed.

Let us mention certain aspects of the techniques used in this paper. The first group of the results is proved mostly by  Dirichlet-Neumann bracketing and it is quite standard approach for proving such kind of results. The second group is established employing the combination of two techniques. The first is the analytic continuation of the resolvent in a vicinity of the threshold of the essential spectrum. The other is a nonsymmetric generalization of Birman-Schwinger principle suggested in \cite{G02}. Such a combination has already been used in \cite{BorExGad02}, \cite{Bor-MSb06} for another model. Here we show that this combination allows also to prove the holomorphy of the emerging eigenvalues w.r.t. $\e$. Moreover, we introduce a new approach of calculating the coefficients for Taylor expansion w.r.t. $\e$ of the emerging eigenvalues. Another advantage of our approach is that it is very useful in studying the same model but in the case of a small width, namely, as $d\to+0$, $\ell=\widetilde{\ell}d$. We found that the aforementioned analytic continuation of the resolvent
can be also used in studying the uniform resolvent convergence in the small-width regime.
In other words, we show that once one can develop such continuation, it is possible to describe the behavior of the emerging eigenvalues and the uniform resolvent convergence in the small-width regime. Studying the small-width case is the subject of the independent paper \cite{Pap2} which is being prepared now and is a continuation of the present paper.

Let us describe briefly the content of the paper. In the next section we formulate the problem and the main results. Third section is devoted to proving general properties of the discrete spectrum. In the forth section we develop the analytic continuation of the resolvent. Employing then this continuation and the nonsymmetric version of Birman-Schwinger principle, in the fifth section we study the emerging eigenvalues.

\section{Formulation of the problem and the main result}\label{Sec.main}
%
Let $x=(x_1,x_2)$ be Cartesian coordinates in $\mathds{R}^2$, and $\Pi$ be an infinite strip of a width $d>0$,
\begin{equation*}
\Pi:=\{x: 0<x_2<d\}.
\end{equation*}
Given a positive number $\ell$, we partition the boundary of $\p\Pi$ into two subsets,
\begin{equation*}
\g_\ell:=\{x: x_1>\ell, x_2=0\}\cup\{x: x_1<-\ell, x_2=d\},\quad \G_\ell:=\p\Pi\setminus\overline{\g}_\ell.
\end{equation*}

The main object of our study is the Laplacian in $\Pi$ subject to Dirichlet boundary condition on $\g_\ell$ and to Neumann one on $\G_\ell$, which we denote as $\Op_\ell$, cf. figure~\ref{pic1}. Rigorously we introduce this operator as the self-adjoint one associated with the symmetric lower-semibounded closed sesquilinear form
\begin{equation}\label{2.3}
\fh(u,v):=(\nabla u, \nabla v)_{L_2(\Pi)}\quad\text{on}\quad \Ho(\Pi,\g_\ell).
\end{equation}
Hereinafter the symbol $\Ho^1(\Om,S)$ indicates the subspace of the functions in $\H^1(\Om)$ vanishing on $S$.

The main goal of this paper is to study the discrete spectrum of $\Op_\ell$ for different values of $\ell$. To formulate the main results, we shall need additional notations. In what follows by $\spec(\cdot)$, $\discspec(\cdot)$, $\essspec(\cdot)$ we denote respectively the spectrum,  the discrete spectrum, and the essential spectrum of an operator. We indicate
\begin{equation*}
E_1:=\frac{\pi^2}{4d^2},\quad
\chi_1(x):=\left\{
\begin{aligned}
&\sqrt{\frac{2}{d}} \sin \frac{\pi}{2d}
 x_2, && x_1>0,
\\
&\sqrt{\frac{2}{d}} \sin  \frac{\pi}{2d}(1-x_2), && x_1<0,
\end{aligned}
\right.
\end{equation*}
and  $\Pi_a:=\Pi\cap\{x: |x_1|<a\}$.

To formulate the main results, we shall employ an auxiliary operator $\Op^*_\ell$. It is the Laplacian in $\Pi$ subject to Dirichlet boundary condition on $\g_\ell^*:=\{x: |x_1|>\ell, x_2=0\}$  and to Neumann one on $\p\Pi\setminus\overline{\g}_\ell^*$. We again introduce it as associated with the form (\ref{2.3}) but on the domain $\Ho^1(\Pi,\g_\ell^*)$, cf. figure~\ref{pic2}. Similar operator but with Dirichlet boundary condition on the upper boundary
of $\Pi$ has already been studied in \cite{Bor-MSb06}. The technique employed in this paper does not use essentially the type of the boundary condition on the upper boundary of $\Pi$. This is why one can reprove all the results of \cite{Bor-MSb06} for the operator $\Op^*_\ell$ up to some minor simple changes related to the Neumann condition on the upper boundary of $\Pi$. We also refer to \cite{DK} where some properties of spectrum of $\Op_\ell^*$ were also studied. We collect the reformulation of the results of \cite{Bor-MSb06} for $\Op_\ell^*$ and the results of \cite{DK} in
\begin{theorem}\label{th2.1}
The spectrum of $\Op_\ell^*$  has the following properties:
\begin{enumerate}
\item\label{it1th2.1} The identity
\begin{equation*}
\essspec(\Op_\ell^*)=[E_1,+\infty)
\end{equation*}
holds true.

\item\label{it2th2.1} There exist infinitely many critical values
\begin{equation*}
0=\ell_1^*<\ell_2^*<\ldots<\ell_n^*<\ldots
\end{equation*}
of the length $\ell$ such that for $\ell\in(\ell_n^*,\ell_{n+1}^*]$ the operator $\Op_\ell^*$ has precisely $n$ isolated eigenvalues $\L_m^*(\ell)$, $m=1,\ldots,n$. These eigenvalues are simple and supposed to be ordered in the ascending order,
\begin{equation*}
\L_1^*(\ell)<\L_2^*(\ell)<\ldots<\L_n^*(\ell).
\end{equation*}

\item\label{it3th2.1}  The number of the eigenvalues $\L_m^*(\ell)$ is estimated as
    \begin{equation*}
    \left[\frac{\ell}{d}\right]\leqslant \#\discspec(\Op_\ell^*)\leqslant \left[\frac{\ell}{d}\right]+1,
    \end{equation*}
    where $[\cdot]$ denotes the entire part of a number. The eigenvalues $\L_m^*(\ell)$ are continuous and non-increasing in $\ell$. They satisfy two-sided estimates
    \begin{equation}\label{2.7}
    \frac{\pi^2(m-1)^2}{4\ell^2} < \L_m^*(\ell) < \frac{\pi^2 m^2}{4\ell^2},\quad m=1,\ldots, n.
    \end{equation}
    The corresponding eigenfunctions are even in $x_1$ for odd $m$ and odd in $x_1$ for even $m$.

\item\label{it4th2.1} The number $\ell$ is critical, if and only if the boundary value problem
    \begin{gather}
    -\Delta \phi_n^*=E_1\phi_n^* \quad\text{in} \quad \Pi,
    \\
    \phi_n^*=0\quad \text{on} \quad \g_\ell^*,\qquad \frac{\p\phi_n^*}{\p x_2}=0\quad \text{on}\quad \p\Pi\setminus\overline{\g_\ell^*},
    \end{gather}
    has a bounded solution belonging to $\H^1(\Pi_a)$ for each $a>0$ and having the asymptotics
    \begin{equation*}
    \phi_n^*(x)=\chi_1(x)+\Odr\big(\E^{-\frac{\sqrt{8}\pi}{d}x_1}\big),\quad x_1\to+\infty.
    \end{equation*}
    If exists such a solution, it is unique and even in $x_1$ for odd $n$ and odd in $x_1$ for even $n$.

\item\label{it5th2.1} As $\ell\to\ell_n^*+0$, $n\geqslant 2$, the eigenvalue $\L_n^*(\ell)$ satisfies the asymptotics
    \begin{equation*}
    \L_n^*(\ell)=E_1- \frac{(\ell-\ell_n^*)^2}{(\ell_n^*)^2} \left( \int\limits_\Pi \left|\frac{\p\phi_n^*}{\p x_1}\right|^2\di x\right)^2+ \Odr\big( (\ell-\ell_n^*)^3\big).
    \end{equation*}
    The associated eigenfunction can be chosen so that
    \begin{align*}
    &\psi_n^*(x)= \E^{-\sqrt{E_1-\L_n^*(\ell)}x_1}\chi_1(x) + \Odr\Big( \E^{-\sqrt{\frac{9\pi^2}{4d^2}-\L_n^*(\ell)} x_1}
    \Big),\quad x_1\to+\infty, 
    \\
    &\psi_n^*(x)=\phi_n^*(x)+\Odr\big((\ell-\ell_n^*)^{1/2}\big) \quad\text{in}\quad \H^1(\Pi_a), 
    \end{align*}
    for each $a>0$.
\end{enumerate}
\end{theorem}

\begin{figure}
\begin{center}\includegraphics[scale=0.6]{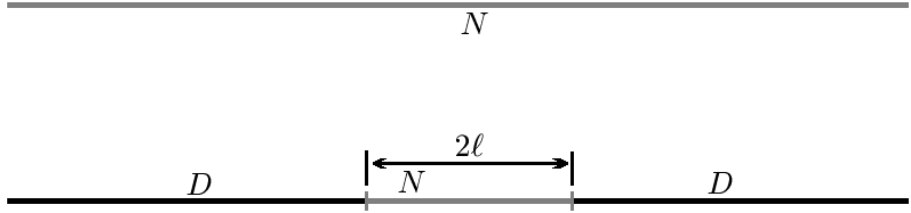}

\caption{The boundary conditions for the auxiliary operator} \label{pic2}
\end{center}
\end{figure}

The main results of this paper are of same fashion as in the last theorem, but for the operator $\Op_\ell$. We gather them in the following theorems. Some of the results are stronger than those in \cite{Bor-MSb06}, \cite{DK}, \cite{BorExGad02}.

Our first main theorem describes general properties of the spectrum of $\Op_\ell$.
\begin{theorem}\label{th2.2}
The spectrum of $\Op_\ell$  has the following properties:
\begin{enumerate}
\item\label{it1th2.2} The identity
\begin{equation*}
\essspec(\Op_\ell)=[E_1,+\infty)
\end{equation*}
holds true.

\item\label{it2th2.2} There exist infinitely many critical values
\begin{equation*}
0<\ell_1 <\ell_2 <\ldots<\ell_n <\ldots
\end{equation*}
of the length $\ell$ such that for $\ell\in(\ell_n,\ell_{n+1}]$ the operator $\Op_\ell$ has precisely $n$ isolated eigenvalues $\L_m(l)$, $m=1,\ldots,n$. These eigenvalues are simple and supposed to be ordered in the ascending order,
\begin{equation}\label{2.14}
\L_1(\ell)<\L_2(\ell)<\ldots<\L_n(\ell).
\end{equation}
The numbers $\ell_m$ satisfy the estimates
\begin{equation}\label{2.15}
\frac{1}{2}\ell_{2m-1}^*\leqslant \ell_m\leqslant \frac{1}{2}\ell_{2m}^*.
\end{equation}

\item\label{it3th2.2}  The number of the eigenvalues $\L_m(\ell)$ is  estimated as
    \begin{equation}\label{2.16}
    \left[\frac{\#\discspec(\Op_{2\ell}^*)}{2}\right]\leqslant \#\discspec(\Op_\ell)\leqslant \left[\frac{\#\discspec(\Op_{2\ell}^*)}{2}\right]+1,
    \end{equation}
    where $[\cdot]$ denotes the entire part of the number. The eigenvalues $\L_m(\ell)$ are real-holomorphic and non-increasing in $\ell$. They satisfy two-sided estimates
    \begin{equation}\label{2.17}
   \L_{2m-1}^*(2\ell)\leqslant \L_m(\ell)\leqslant \L_{2m}^*(2\ell),\quad m=1,\ldots, n.
    \end{equation}
    The corresponding eigenfunctions  $\psi_m(x,\ell)$ are even w.r.t. the symmetry transformation
    \begin{equation}\label{2.7a}
    (x_1,x_2)\mapsto (-x_1,d-x_2)
    \end{equation}
    for odd $m$ and odd for even $m$, i.e.,
    \begin{equation}\label{2.19a}
    \begin{aligned}
    &\psi_m(-x_1,d-x_2)=\psi_m(x)&& \text{for odd $m$},
    \\
    &\psi_m(-x_1,d-x_2)=-\psi_m(x)&& \text{for even $m$}.
    \end{aligned}
    \end{equation}

\end{enumerate}
\end{theorem}

For each $a>0$ we let
$\Pi_a^\pm:=\Pi\cap\{x: \pm x_1>a\}$.

The second theorem is devoted to the emergence of new eigenvalues from the threshold of the essential spectrum as $\ell$ is close to a critical value.

\begin{theorem}\label{th2.3}
The behavior of the eigenvalues of $\Op_\ell$ in a vicinity of the threshold of the essential spectrum is described by the following statements:
\begin{enumerate}
\item\label{it4th2.2} The number $\ell=\ell_n$ is critical, if and only if the boundary value problem
    \begin{equation}\label{2.18}
    \begin{gathered}
    -\Delta \phi_n=E_1\phi_n \quad\text{in} \quad \Pi,
    \\
    \phi_n=0\quad \text{on} \quad \g_{\ell_n},\qquad \frac{\p\phi_n}{\p x_2}=0\quad \text{on}\quad \G_{\ell_n}
    \end{gathered}
    \end{equation}
    has a bounded solution belonging to $\H^1(\Pi_a)$ for each $a>0$ and having the asymptotics
    \begin{equation}\label{2.19}
    \phi_n(x)=\chi_1(x)+\Odr\big(\E^{-\frac{\sqrt{8}\pi}{d}x_1}\big),\quad x_1\to+\infty.
    \end{equation}
    This solution is unique. For even $n$ it is odd w.r.t.
    the symmetry transformation (\ref{2.7a}), and is even for odd $m$.

\item\label{it5th2.2} Let $\e:=\ell-\ell_n$. For sufficiently small $\e$ the eigenvalue $\L_n(\ell)$ is given by the formula
    \begin{equation}\label{2.31}
    \L_n(\ell)=E_1-\mu_n^2(\e).
    \end{equation}
    Here $\mu_n(\e)$ is a real-holomorphic function represented as a convergent series
    \begin{align}
    &\mu_n(\e)=\sum\limits_{j=1}^{\infty}\e^j \mu_j^{(n)},\label{2.20}
    \\
    &\mu_1^{(n)}=\frac{1}{\ell_*} \int\limits_{\Pi} \left|\frac{\p \phi_n}{\p x_1}\right|^2\di x,
    \label{2.24}
    \\
    &
    \begin{aligned}
    \mu_2^{(n)}=-\frac{(\mu_1^{(n)})^2}{2} \Big(& \|\phi_n\|_{L_2(\Pi_{\ell_*})}^2
    + \|\phi_n-\chi_1\|_{L_2(\Pi_{\ell_*}^+)}
    \\
    & +\|\phi_n-\wp_n\chi_1\|_{L_2(\Pi_{-\ell_*}^-)}
    \Big) +  \frac{1}{\ell_*} \int\limits_{\Pi} \frac{\p\phi_n}{\p x_1} \frac{\p\ph_n}{\p x_1}\di x,
    \end{aligned}    \label{2.29}
    \end{align}
    where $\phi_n$ is the unique solution to the problem
    \begin{equation}\label{2.30}
    \begin{aligned}
    &-\Delta \ph_n=E_1\ph_n \quad\text{in} \quad \Pi,
    \quad \ph_n=0\quad \text{on} \quad \g_{\ell_n},\quad \frac{\p\ph_n}{\p x_2}=0\quad \text{on}\quad \G_{\ell_n},
    \\
    &\ph_n(x)=-\mu_1^{(n)}\, x_1\chi_1(x)
    +\Odr\big(\E^{-\frac{\sqrt{8}\pi}{d}|x_1|}\big), \quad x_1\to+\infty,
    \\
    &\ph_n(x)=\wp\, \mu_1^{(n)} x_1\chi_1(x) +\Odr\big(\E^{-\frac{\sqrt{8}\pi}{d}|x_1|}\big), \quad x_1\to-\infty,
    \\
    &\ph_n(x)=
    \frac{2(\mu_1^{(n)})^{\frac{1}{2}}}{\pi^{\frac{1}{2}}} r^{-\frac{1}{2}}\sin \frac{\tht}{2} 
    +\Odr(r^{\frac{1}{2}}), \quad x\to(\ell_n,0),
    \\
    &\ph_n(-x_1,d-x_2)=(-1)^{n-1}\ph_n(x),
    \end{aligned}
    \end{equation}
    where $(r,\tht)$ are the polar coordinates centered at $(\ell_n,0)$. Other coefficients of the series (\ref{2.20}) are given by the formulas (\ref{4.26}).

\end{enumerate}
\end{theorem}

In the proof of the item~\ref{it5th2.2} of the last theorem we also show that the eigenfunction $\phi_n$ associated with $\L_n$ is holomorphic w.r.t. $\e$ up to a special change of variables. We do not give this result here since it requires additional notations and we refer to Sec.~\ref{sec:emev}, where this result is presented in all the details -- see the series (\ref{5.11}) and the construction of its coefficients.

The item~\ref{it5th2.2} of the last theorem is much stronger than similar results established in \cite{BorExGad02}, \cite{Bor-MSb06}. Namely, in these papers only two-terms asymptotics for the emerging eigenvalues were obtained. Our results provides a complete asymptotic expansion (\ref{2.31}), (\ref{2.20}). Moreover, we also prove that this series converges to the eigenvalue and is holomorphic w.r.t. a small parameter $\e$. The formula for the third term asymptotics was absent in \cite{BorExGad02}, \cite{Bor-MSb06}, too. In our case it is the identity (\ref{2.29}). The same concerns the eigenfunction; in \cite{BorExGad02}, \cite{Bor-MSb06} only the rate of the convergence of the eigenfunctions associated with emerging eigenvalues was  estimated. We also observe that the holomorphy of the eigenvalues w.r.t. the window's length was not proven in \cite{BorExGad02}, \cite{Bor-MSb06}. In our case we establish such property for the eigenvalues in item~\ref{it3th2.2} of Theorem~\ref{th2.2}.

\section{Essential spectrum, existence and continuity of the eigenvalues}

In this section we prove the general properties of the spectrum of $\Op_\ell$ claimed in Theorem~\ref{th2.2}. We begin  with item~\ref{it1th2.2} of this theorem.

\begin{lemma}\label{lm3.1}
The item~\ref{it1th2.2} of Theorem~\ref{th2.2} is valid.
\end{lemma}
\begin{proof}
The essential spectra of $\Op_\ell$ and $\Op_0$ are same that can be proven completely by analogy with Theorem~2.1 in \cite{BorEkhKov} reproducing the ideas of \cite{Bi}. And it is easy to check that $\essspec(\Op_0)=\essspec(\Op_\ell)= [E_1,+\infty)$.
\end{proof}

Our next step is the estimates (\ref{2.17}).

\begin{lemma}\label{lm3.2}
The estimates (\ref{2.17}) hold true.
\end{lemma}
\begin{proof}
We prove the formulas by standard bracketing arguments \cite[Ch. X\!I\!I\!I, Sec. 15, Prop. 4]{RS}. In the strip $\Pi$ we introduce an additional boundary $\ga:=\{x: x_1=\ell, 0<x_2<d\}$. Consider the operator $\Op_\ell$ and impose Dirichlet or Neumann boundary condition on $\ga$. It increases or decreases the eigenvalue of $\Op_\ell$. This is the main idea of the proof and let us perform it in all the details.

We impose Dirichlet condition on $\ga$ and consider the Laplacians $\Op_\ell^-$ and $\Op_\ell^+$ in $\Pi_{-\ell}^-$ and $\Pi_\ell^+$ subject to Dirichlet boundary condition on $\ga$ and to the same condition on $\p\Pi$ as in the operator $\Op_\ell$. Rigorously we define these operators as associated with the sesquilinear form in (\ref{2.3}) but on $\Ho^1(\Pi_{\pm\ell}^\pm,\g^\pm_\ell)$, $\g^\pm_\ell:=\ga\cup\big(\p\Pi^\pm_{\pm\ell}\cap\g_\ell\big)$.

In the same way as in Lemma~\ref{lm3.1} one can check that
\begin{equation*}
\essspec(\Op_\ell^-)=[E_1,+\infty),
\quad \spec(\Op_\ell^+)=\essspec(\Op_\ell^+)=[E_1,+\infty).
\end{equation*}

Consider now the eigenfunctions of $\Op_\ell^-$. We extend them in an odd way via the boundary $\ga$ and see that up to the shift $x_1\mapsto x_1-2\ell$ the extensions are exactly the eigenfunctions of $\Op^*_{2\ell}$ being odd w.r.t. $x_1$. And vice versa, each odd eigenfunction of $\Op^*_{2\ell}$ generates an eigenfunction of $\Op_\ell$. Bearing in mind the item~\ref{it3th2.1} of Theorem~\ref{th2.1}, we conclude that the isolated eigenvalues of $\Op_\ell^-$ are exactly $\L_{2m}^*(2\ell)$. Thus,
\begin{equation*}
\essspec(\Op_\ell^-\oplus\Op_\ell^+)=[E_1,+\infty),
\end{equation*}
and the isolated eigenvalues of $\Op_\ell^-\oplus\Op_\ell^+$ are $\L_{2m}^*(2\ell)$. At the same time, by the bracketing arguments the eigenvalues of $\Op_\ell^-\oplus\Op_\ell^+$ give the upper bounds for the eigenvalues of $\Op_\ell$. It proves the right part of the estimates (\ref{2.17}). The left part is proven completely in the same way, just on $\ga$ one should impose the Neumann condition.
\end{proof}

Now we are in position to prove the item~\ref{it3th2.2} of Theorem~\ref{th2.2}.
\begin{lemma}\label{lm3.3}
The estimates (\ref{2.16}), (\ref{2.17}) are valid. The eigenvalues $\L_m$ are simple. The corresponding eigenfunctions have certain parity w.r.t. the symmetry transformation (\ref{2.7a}).
\end{lemma}
\begin{proof}
The estimates (\ref{2.16}) follow directly from (\ref{2.17}). By (\ref{2.7}) we have $\L_m^*(\ell)\not=\L_{m+1}^*(\ell)$ for all $m$ and $\ell$, and hence the estimates (\ref{2.17}) also imply that $\L_m(\ell)$ are simple for all $m$ and $\ell$. Since the sets $\Pi$, $\g_\ell$, and $\G_\ell$ are symmetric under the transformation (\ref{2.7a}) and all the eigenvalues $\L_m$ are simple, we conclude that all the eigenfunctions of $\Op_\ell$ have certain parity under (\ref{2.7a}). The eigenvalues $\L_m(\ell)$ are non-increasing in $\ell$ by the bracketing arguments \cite[Ch. X\!I\!I\!I, Sec. 15, Prop. 4]{RS}, since as $\ell$ increases the set $\g_\ell$ increases and $\G_\ell$ decreases. The eigenvalues $\L_m(\ell)$ are also holomorphic in $\ell$ since the resolvent of $\Op_\ell$ is holomorphic in $\ell$. This fact is proven by analogy with Lemma~2.2 in \cite{Bor-MSb06} by a simple change of variables
\begin{equation}\label{5.1}
x\mapsto y, \quad y=(y_1,y_2),\quad y_1=x_1-(\ell-\ell_*)\xi_1(x_1),\quad y_2=x_2.
\end{equation}
Here $\ell_*$ is a point at which we prove the holomorphy of $\Op_\ell$ and $\xi_1=\xi_1(x_1)$ is an infinitely differentiable cut-off function equalling $\pm 1$ as $|x_1\mp\ell_*|<\ell_*/3$ and vanishing as $|x_1\mp\ell_*|>2\ell_*/3$.
\end{proof}

\begin{lemma}\label{lm3.4}
The identities (\ref{2.19a}) hold true.
\end{lemma}

\begin{proof}
We represent the operator $\Op_\ell$ as a direct sum $\Op_\ell=\widetilde{\Op}_\ell^{\mathrm{even}}\oplus \widetilde{\Op}_\ell^{\mathrm{odd}}$, where $\widetilde{\Op}_\ell^{\mathrm{even}}$ is the restriction of $\Op_\ell$ on the functions being even w.r.t. (\ref{2.7a}), while $\widetilde{\Op}_\ell^{\mathrm{odd}}$ is the restriction on the odd functions. Due to Lemma~\ref{lm3.3} all the eigenfunctions $\psi_m$ are also those of either $\widetilde{\Op}_\ell^{\mathrm{even}}$  or $\widetilde{\Op}_\ell^{\mathrm{odd}}$. The same is thus true for the eigenvalues.

Consider the operator $\widetilde{\Op}_\ell^{\mathrm{even}}$. We introduce in $\Pi$ additional boundaries $\{x: x_1=\pm \ell, \ 0<x_2<d\}$ and impose there Neumann or Dirichlet boundary condition. As in the proof of Lemma~\ref{lm3.2}, by bracketing arguments it implies  upper and lower bounds for the eigenvalues $\widetilde{\L}_p^+(\ell)$ of $\widetilde{\Op}_\ell^{\mathrm{even}}$:
\begin{equation}\label{3.3}
\frac{\pi^2 p^2}{\ell^2}\leqslant \widetilde{\L}_p^+(\ell)\leqslant
\frac{\pi^2 (2p+1)^2}{4\ell^2},\quad p=0,1,\ldots
\end{equation}
In the same way one can get similar estimates for the eigenvalues of $\widetilde{\Op}_\ell^{\mathrm{odd}}$:
\begin{equation}\label{3.4}
\frac{\pi^2(2p-1)^2}{4\ell^2}\leqslant \widetilde{\L}_p^-(\ell)
\leqslant \frac{\pi^2 p^2}{\ell^2},\quad p=1,2,\ldots
\end{equation}
It follows from (\ref{2.7}), (\ref{2.17}) that each of the intervals
\begin{equation*}
\left(\frac{\pi^2(m-1)^2}{4\ell^2},\frac{\pi^2 m^2}{\ell^2}\right),\quad m=1,2,\ldots
\end{equation*}
contains exactly one eigenvalue $\L_m(\ell)$. At the same time, by (\ref{3.3}), (\ref{3.4}) the eigenfunctions associated with these eigenvalues are even under (\ref{2.7a}) for odd $m$ and are odd for even $m$. This completes the proof.
\end{proof}

The rest of this section is devoted to the proof of item~\ref{it2th2.2} of Theorem~\ref{th2.2}. The simplicity of $\L_m(\ell)$ is due to Lemma~\ref{lm3.3}. It also yields (\ref{2.14}).

The existence of $\ell_1>0$ such that $\discspec(\Op_\ell)=\emptyset$ for $\ell<\ell_1$ and $\discspec(\Op_\ell)\not=\emptyset$ for $\ell>\ell_1$ was proven in \cite[Th. 3]{DK}. Since the lowest eigenvalue $\L_1(\ell)$ is monotone in $\ell$, there exists a limit $c:=\lim\limits_{\ell\to\ell_1+0} \L_1(\ell)$. If $c<E_1$, the number $c$ should be an isolated eigenvalue of $\Op_{\ell_1}$. At the same time, the discrete spectrum of $\Op_{\ell_1}$ is empty and hence $c=E_1$. Moreover, it follows from (\ref{2.17}) that for $\ell<\ell_1$ the discrete spectrum of $\Op_{2\ell}^*$ contains at most one eigenvalue. Hence, by the item~\ref{it2th2.1} of Theorem~\ref{th2.1} we have
\begin{equation*}
\ell_1\leqslant \frac{1}{2}\ell_2^*.
\end{equation*}

We prove the existence of the other values $\ell_n$, $n\geqslant 2$, and their upper bounds from (\ref{2.15}) by the induction. Suppose we have proven the existence of $\ell_n$, $n\leqslant k-1$, such that for $\ell\leqslant \ell_n$ the operator $\Op_\ell$ has precisely $(n-1)$
isolated eigenvalues, while for $\ell>\ell_n$ it has at least
$n$ isolated eigenvalues. The estimates (\ref{2.17}) and the item~\ref{it2th2.1} of Theorem~\ref{th2.1} imply that for sufficiently large $\ell$, namely, for $\ell>\frac{1}{2}\ell_{2k}^*$ the operator $\Op_\ell$ has at least $k$ isolated eigenvalues. Consider the $k$-th eigenvalue $\L_k(\ell)$ for values $\ell$ when $\L_k(\ell)$ exists. Since $\L_k$ is non-increasing and continuous in $\ell$, if it exists for some $\ell_*$, it exists also for all $\ell\geqslant \ell_*$. Hence, again by the continuity and monotonicity of $\L_k$ we conclude that there exists $\ell_k$ such that
\begin{equation*}
\lim\limits_{\ell\to\ell_k+0} \L_k(\ell)=E_1,
\end{equation*}
and moreover
\begin{equation*}
\ell_{k-1}<\ell_k\leqslant \frac{1}{2} \ell_{2k}^*.
\end{equation*}
Otherwise the operator $\Op_\ell$ would have $k$ isolated eigenvalues for $\ell\leqslant \ell_{k-1}$ which contradicts to the assumption of the induction. As $\ell\leqslant \ell_k$,
the spectrum of $\Op_\ell$ contains precisely $(k-1)$ isolated eigenvalues, while for $\ell>\ell_k$ it has at least $k$ eigenvalues. Therefore, we proven the existence of $\ell_n$ and their upper bounds in (\ref{2.15}). The lower ones follow from the existence of $\ell_n$,  and Lemma~\ref{lm3.3}.

Given any point $\ell_*$, we again employ the change (\ref{5.1}) and see that after the change the resolvent of the operator is holomorphic in a vicinity of $\ell_*$. Hence, all the eigenvalues of $\Op_\ell$ which remain isolated for $\ell=\ell_*$ are holomorphic in a vicinity of $\ell_*$. Theorem~\ref{th2.2} is proven.

\section{Behavior of the resolvent at the threshold of the essential spectrum}

In this section we study the behavior of the resolvent $(\Op_\ell-\l)^{-1}$ as $\l$ is close to the threshold $E_1$ of the essential spectrum $\essspec(\Op_\ell)$. It is the core for studying the behavior of the eigenvalues $\L_m(\ell)$ in a vicinity of the threshold $E_1$ as $\ell\to\ell_m+0$. It will be also a core in studying the case of a small-width regime in \cite{Pap2}, as it was already mentioned in the introduction.

The technique presented  in this section is based on a special analytic continuation of the resolvent $(\Op_\ell-\l)^{-1}$ w.r.t. the spectral parameter and a scheme described in \cite[Ch. X\!V\!I, Sec. 4]{S-P} and originally came from A.~Majda \cite{M74}. Such a combination has already been used successfully in \cite{Bor-MSb06}, \cite{BorExGad02}, \cite{G02} and in this section we extend this technique to our model.

Given $\ell$, we take a fixed number $a>\ell$ and consider $f\in L_2(\Pi_a)$. In what follows all the considered functions from $L_2(\Pi_a)$ are assumed to be extended by zero in $\Pi\setminus\Pi_a$.

Let $u$ be a generalized solution to the boundary value problem
\begin{equation}\label{v3.6}
(-\D-E_1+\mu^2)u=f \quad\text{in}\quad \Pi,\qquad u=0\quad\text{on}\quad\g_\ell,\qquad \frac{\p u}{\p x_2}=0\quad \text{on} \quad \G_\ell,
\end{equation}
behaving at infinity as follows,
\begin{equation}\label{4.0}
u(x)=c_\pm(\mu)\E^{\mp\mu x_1}\chi_1(x)+\Odr(\E^{-\sqrt{E_2-E_1+\mu^2}|x_1|}),\quad x_1\to\pm\infty.
\end{equation}
Here $\mu$ is a small complex number, $c_\pm$ are some constants, and
\begin{equation*}
E_m:=\frac{\pi^2}{d^2}\left(m-\frac{1}{2}\right)^2,\quad
\chi_m(x):=\left\{
\begin{aligned}
&\sqrt{\frac{2}{d}} \sin\sqrt{E_m} x_2, && x_1>0,
\\
&\sqrt{\frac{2}{d}} \sin\sqrt{E_m}(1-x_2), && x_1<0,
\end{aligned}
\right.
\end{equation*}

Let $g$ be a function from $L_2(\Pi_{l})$. We introduce the function $v$ as the solution to the boundary value problem
\begin{equation}\label{v3.6a}
\begin{gathered}
(-\D-E_1+\mu^2)v=g\quad\text{in}\quad\Pi,
\\
v=0\quad\text{on}\quad \g_0\cup\{x: x_1=0, 0<x_2<d\},\qquad \frac{\p v}{\p y_2}=0\quad \text{on}\quad \G_0.
\end{gathered}
\end{equation}
The solution is given by the formulas
\begin{align*}
&v(x,\mu)=\sum\limits_{m=1}^{\infty} V_m(x,\mu)\chi_n(x),
\\
&V_m(x,\mu)=\left\{
\begin{aligned}
& \int\limits_0^{+\infty} \frac{\E^{-k_m(\mu)|x_1-t_1|}-\E^{-k_m(\mu)(x_1+t_1)}}{2k_m(\mu)} g_m^{(+)}(t_1)\di t_1,\quad x_1>0,
\\
&  \int\limits_{-\infty}^0 \frac{\E^{-k_m(\mu)|x_1-t_1|}-\E^{k_m(\mu)(x_1+t_1)}}{2 k_m(\mu)} g_m^{(-)}(t_1)\di t_1,\quad x_1<0,
\end{aligned}
\right.
\\
&g_m^{(\pm)}(y_1):=\int\limits_0^1 \chi_m(x) g(x)\di x_2,\quad \pm x_1>0,
\\
&k_1(\mu):=\mu,\quad k_m(\mu):=\sqrt{E_m-E_1-\mu^2},\quad m\geqslant 2,
\end{align*}
where the branch of the square root is fixed by the requirement $\sqrt{1}=1$.  We indicate the mapping $g\mapsto v$ by $\cc{T}_1(\mu)v$.

\begin{lemma}\label{vlm3.1}
The mapping $\cc{T}_1(\mu)$ is a bounded linear operator from $L_2(\Pi_a)$ into
\begin{equation*}
\Big(\H^2(\Pi_a\cap\{x: x_1>0\})\oplus \H^2(\Pi_a\cap\{x: x_1<0\})\Big)\cap \H^1(\Pi_a).
\end{equation*}
The operator $\cc{T}_1(\mu)$ is holomorphic w.r.t. sufficiently small complex $\mu$. As $\mu=0$, the operator $\cc{T}_1(0)$ acts as follows
\begin{align}
&(\cc{T}_1(0)g)(x)=\sum\limits_{m=1}^{\infty} V_m(x,0)\chi_n(x), \nonumber 
\\
&V_1(x,0)=-\left\{
\begin{aligned}
& \int\limits_0^{+\infty} \frac{|x_1-t_1|-(x_1+t_1)}{2} g_1^{(+)}(t_1)\di t_1,\quad x_1>0,
\\
&  \int\limits_{-\infty}^0 \frac{|x_1-t_1|+(x_1+t_1)}{2} g_1^{(-)}(t_1)\di t_1,\quad x_1<0.
\end{aligned}
\right.\nonumber
\end{align}
\end{lemma}

The proof of this lemma repeats word-by-word that of Lemma~3.1 in \cite{Bor-MSb06}. The only minor change is the basis for the Fourier series on the cross-section of $\Pi$ which is $\{\chi_m\}_{m=1}^\infty$ in our case.

Let $v:=\cc{T}_1(\mu)g$ and consider one more boundary value problem
\begin{equation}\label{v3.9}
(-\D+1) w=(-\D+1)v\quad\text{in}\quad \Pi_a,\qquad w=v\quad\text{on}\quad \g_{\ell,a},\qquad \frac{\p w}{\p y_2}=0\quad \text{on}\quad\G_{\ell,a},
\end{equation}
where
\begin{align*}
&\g_{\ell,a}:=(\g_\ell\cap\p\Pi_a)\cup\g_+\cup\g_-,\quad
\G_{\ell,a}:=\G_\ell\cap\p\Pi_a,
\\
&\g_\pm=\{x: x_1=\pm a, 0<x_2<d\}.
\end{align*}
Since the function $v$ does not belong to $\H^2(\Pi_a)$, but it is in $\H^2(\Pi_a\cap \{x: x_1>0\})$ and
$\H^2(\Pi_a\cap \{x: x_1<0\})$, the right hand side in (\ref{v3.9}) is treated pointwise. Namely, we consider  the function $(-\D+1)v$ separately for $x_1<0$ and $x_1>0$.  Under such definition it belongs to $L_2(\Pi_a)$.

The problem (\ref{v3.9}) is obviously solvable in $\H^1(\Pi_a)$. We denote the mapping $g\mapsto w$ as $\cc{T}_2(\mu)$.
By $\Xi$ we denote the set of all bounded subdomains of $\Pi$ separated from the points $(\ell,0)$ and $(-\ell,d)$ and having $C^2$-boundaries.

\begin{lemma}\label{vlm3.2}
Let $Q\in \Xi$. The mapping $\cc{T}_2(\mu)$ is a linear bounded operator from $L_2(\Pi_a)$ into $\H^1(\Pi_a)\cap\H^2(Q)$. This operator is holomorphic in $\mu$.
\end{lemma}
\begin{proof}
The set $\Pi_a$ possesses the cone property and by Rellich-Kondrashov theorem \cite[Ch. V\!I, Thm. 6.2]{Ad}
the embedding $\H^1(\Pi_a)\subset L_2(\Pi_a)$ is compact. Thus, the problem (\ref{v3.9})  is uniquely solvable in $\H^1(\Pi_a)$. By the standard smoothness improving theorems the mapping $v\mapsto w$ is a bounded linear operator from
\begin{equation*}
\Big(\H^2(\Pi_a\cap\{x: x_1>0\})\oplus \H^2(\Pi_a\cap\{x: x_1<0\})\Big)\cap \H^1(\Pi_a).
\end{equation*}
into $\H^1(\Pi_a)\cap\H^2(Q)$. The composition of this operator with $\cc{T}_1(\mu)$ is exactly the operator $\cc{T}_2(\mu)$. In view of Lemma~\ref{vlm3.1} it completes the proof.
\end{proof}

The main idea of the aforementioned scheme borrowed from \cite{S-P}, \cite{M74} is to construct $u$ as
\begin{equation}\label{v3.10}
u(x,\mu)=\cc{T}_3(\mu)g:=\xi_2(x_1) v(x,\mu)+\big(1-\xi_2(x_1)\big) w(x,\mu),
\end{equation}
where $\xi_2=\xi_2(x_1)$ is an infinitely differentiable cut-off function equalling one as $|x_1|>\frac{1}{3}l+\frac{2}{3}a$ and vanishing as $|y_1|<\frac{2}{3}l+\frac{1}{3}a$.

We shall also need to know certain properties of the operator $\cc{T}_3(\mu)$.  All of them are collected in
\begin{lemma}\label{vlm3.4}
Let $b$ be a positive fixed number. The mapping $\cc{T}_3(\mu)$ is a bounded linear operator from $L_2(\Pi_a)$ into $\H^1(\Pi_b)$ and is holomorphic w.r.t. small complex $\mu$.
\end{lemma}

The statement of this lemma follows directly from Lemmas~\ref{vlm3.1},~\ref{vlm3.2} and the definition (\ref{v3.11}) of $\cc{T}_3$.

The function $\cc{T}_3(\mu)g$ satisfies the boundary condition in (\ref{v3.6}). Hence, to be a solution to (\ref{v3.6}), (\ref{4.0}) the function $\cc{T}_3(\mu)g$ should satisfy only the equation in (\ref{v3.6}). We substitute $\cc{T}_3(\mu)g$ into this equation and employ that in (\ref{v3.6a}), (\ref{v3.9}) for $v$ and $w$,
\begin{align}
&(-\D-E_1+\mu^2) \cc{T}_3(\mu)g= (-\D-E_1+\mu^2)(\xi_2 v+(1-\xi_2)w)\nonumber
\\
&\hphantom{(-\D-E_1+\mu^2) \cc{T}_3(\mu)g}=g+\cc{T}_4(\mu)g,\nonumber
\\
&\cc{T}_4(\mu)g:=2\nabla\xi_2\cdot\nabla (w-v)+(w-v) (-\D-E_1+\mu^2-1)(1-\xi_2).
\label{v3.11}
\end{align}
Thus, the function $g$ should be a solution to
\begin{equation}\label{v3.12}
g+\cc{T}_4(\mu)g=f.
\end{equation}
This equation is considered in $L_2(\Pi_a)$ and is equivalent to the problem (\ref{v3.6}), (\ref{4.0}). This fact is stated by
\begin{lemma}\label{vlm3.3}
The operator $\cc{T}_4(\mu)$ is a compact linear one in $L_2(\Pi_a)$ and is holomorphic w.r.t. small complex $\mu$. The equation (\ref{v3.12}) is equivalent to the problem (\ref{v3.6}), (\ref{4.0}) for small complex $\mu$ in the following sense. For each function $f\in L_2(\Pi_a)$, the solution $g$ to (\ref{v3.12}) generates the solution to (\ref{v3.6}), (\ref{4.0}) by the formula $u=\cc{T}_3(\mu)g$. And vice versa, given a solution $u$ to (\ref{v3.6}), (\ref{4.0}), there exists a solution $g$ to (\ref{v3.12}) such that $u=T_3(\mu)g$.
\end{lemma}
\begin{proof}
Since the function $\xi_2$ is non-constant only for
\begin{equation*}
\ell<\frac{2\ell+a}{3}<|x_1|<\frac{\ell+2a}{3},
\end{equation*}
it follows from the definition (\ref{v3.11}) of $\cc{T}_4$ that this operator considered as depending on $(w-v)$ is a bounded holomorphic in $\mu$ mapping from $\H^1(\Pi_{(\ell+2a)/3})\setminus\Pi_{(2\ell+a)/3})$ into $L_2(\Pi_a)$. Lemmas~\ref{vlm3.1},~\ref{vlm3.2} is a linear bounded operator from $L_2(\Pi_a)$ into $\H^2(\Pi_{(\ell+2a)/3})\setminus\Pi_{(2\ell+a)/3})$. Hence, due to the compact embedding of $\H^2(\Pi_{(\ell+2a)/3})\setminus\Pi_{(2\ell+a)/3})$
into $\H^1(\Pi_{(\ell+2a)/3})\setminus\Pi_{(2\ell+a)/3})$ the operator $\cc{T}_4$ is compact in $L_2(\Pi_a)$ and holomorphic in small complex $\mu$.

Given $f$, we take a solution $g$ to (\ref{v3.12}) and construct the associated one to (\ref{v3.6}), (\ref{4.0}) as $u=\cc{T}_3(\mu)g$.

Let us show how to determine the solution to (\ref{v3.12}) once we know that of (\ref{v3.6}), (\ref{4.0}).
We do it by analogy with Propositions~3.1,~3.2 in \cite{BorExGad02}. Given $f$ and the associated solution to (\ref{v3.6}), (\ref{4.0}), consider the boundary value problem
\begin{align*}
&(-\D+1)U=0\quad\text{in}\quad\Pi_a,
\\
&\frac{\p U}{\p x_2}=0\quad\text{on}\quad\G_{0,a},\qquad U=0\quad\text{on}\quad\g_{\ell,a},
\\
&U=u\quad\text{on}\quad \big(\G_{\ell,a}\setminus\overline{\G}_{0,a}\big)\cup\{x: x_1=0, 0<x_2<d\}.
\end{align*}
This problem is uniquely solvable in $\H^1(\Pi_a)$. We define $v:=u-(1-\xi_2) U$. It is easy to see that this function belongs to
\begin{equation*}
\H^1(\Pi_a)\cap \big(\H^2(\Pi_a\cap\{x: x_1<0\})\oplus \H^2(\Pi_a\cap\{x: x_1>0\})\big)
\end{equation*}
and solves the boundary value problem (\ref{v3.6a}) with $g:=(-\D-E_1+\mu^2)v$. Here the function $g$ is defined separately for $x_1>0$ and $x_1<0$.

The introduced function $v$ is exactly the sought solution to the equation (\ref{v3.12}); the corresponding function $w$ is $w:=u+\xi_2 U$. These facts and the boundary value problem (\ref{v3.9}) for $w$ are checked by direct calculations.
\end{proof}

Since the operator $\cc{T}_4$ is compact, we apply the Fredholm alternatives to (\ref{v3.12}) and by Lemma~\ref{vlm3.3} it describes the solvability of (\ref{v3.6}), (\ref{4.0}). As $\RE\mu>0$, the solution to (\ref{v3.6}), (\ref{4.0}) decays exponentially as $x_1\to\pm\infty$ and belongs to $\H^1(\Pi)$. In this case we have
\begin{equation*}
\left(\Op_\ell -E_1+\mu^2\right)^{-1}f=\cc{T}_3(\mu)(\I+\cc{T}_4(\mu))^{-1}f.
\end{equation*}
In view of this identity it is natural to treat the operator $\cc{T}_3(\mu)(\I+\cc{T}_4(\mu))^{-1}$ as an analytic in $\mu$ continuation of $\left(\Op_\ell -E_1+\mu^2\right)^{-1}$ considered on $L_2(\Pi_a)$.

Completely by analogy with Lemmas~3.4,~3.5 in \cite{Bor-MSb06} we prove

\begin{lemma}\label{lm4.1}
For sufficiently small $\mu$ the operator $(\I+\cc{T}_4(\mu))^{-1}$ is well-defined in $L_2(\Pi_a)$. It is bounded and meromorphic w.r.t. $\mu$.
\end{lemma}

The last lemma implies that the only possible singularities to $(\I+\cc{T}_4(\cdot))^{-1}$ are isolated poles. Since we consider this operator only for small complex $\mu$, we are interested in a possible pole at zero. The presence of such pole is equivalent to the non-invertibility of $(\I+\cc{T}_4(0))$. In its turn, it yields that in this case the equation (\ref{v3.12}) has a nontrivial solution as $\mu=0$, $f=0$.  By Lemma~\ref{vlm3.3} it is equivalent to the existence of a non-trivial solution to (\ref{v3.6}), (\ref{4.0}) with $f=0$, $\mu=0$. In the next lemma we describe the only possible nontrivial solution to (\ref{v3.6}), (\ref{4.0}) with $f=0$, $\mu=0$.

\begin{lemma}\label{lm4.2}
The problem (\ref{v3.6}), (\ref{4.0}) with $f=0$, $\mu=0$ has at most one nontrivial solution up to a multiplicative constant.  If it exists, the corresponding constants $c_\pm(0)$ in (\ref{4.0}) are non-zero and can be chosen so that $c_+(0)=1$ and $c_-(0)$ is either $1$ or $-1$. This nontrivial solution has a certain parity under the transformation (\ref{2.7a}).
\end{lemma}

\begin{proof}
We prove first that the problem (\ref{v3.6}), (\ref{4.0}) can not possess a nontrivial exponentially decaying solution. Suppose such solution exists and denote it by $u$. By the standard smoothness improving theorems this function is infinitely differentiable everywhere in $\overline{\Pi}$ except the points $(\ell,0)$ and $(-\ell,d)$ and is continuous in $\overline{\Pi}$. Hence, the function
\begin{equation*}
U(x):=\int\limits_{-\infty}^{x_1} t u(t,x_2)\di t
\end{equation*}
is well-defined in $\overline{\Pi}$ and has the same smoothness as $u$. Employing (\ref{v3.6}), by direct calculations we check that
\begin{align}
&
\begin{aligned}
(\D+E_1) U(x)&=\frac{\p}{\p x_1} \big( x_1 u(x)\big)+\int\limits_{-\infty}^{x_1} \left(\frac{\p^2}{\p x_2^2}+E_1\right) t u(t,x_2)\di t
\\
=&\int\limits_{-\infty}^{x_1} (\D_{t,x_2}+E_1) t u(t,x_2)\di t
=2\int\limits_{-\infty}^{x_1} \frac{\p u}{\p t}(t,x_2)\di t
=2u(x),\quad x\in\Pi,
\end{aligned} \nonumber
\\
&
\begin{aligned}
&U=0&&\text{on}\quad \g_\ell^-:=\{x: x_1<-\ell, x_2=d\},
\\
&U=\int\limits_{-\infty}^{\ell} t u(t,0)\di t &&\text{on}\quad \g_\ell^+:=\{x: x_1>\ell, x_2=0\},
\\
&\frac{\p U}{\p x_2}=0&&\text{on}\quad \G_\ell^-:=\{x: x_1<\ell, x_2=0\},
\\
&\frac{\p U}{\p x_2}=\int\limits_{-\infty}^{-\ell} t \frac{\p u}{\p x_2}(t,0)\di t &&\text{on}\quad \G_\ell^+:=\{x: x_1>-\ell, x_2=d\}.
\end{aligned}
\end{align}
Since the function $u$ decays exponentially as $x_1\to\pm\infty$, the function $U$ decays exponentially as $x_1\to-\infty$ and is bounded as $x_1\to+\infty$.

We use the established properties of $U$ and the eigenvalue problem for $u$ and integrate by parts as follows,
\begin{equation*}
2\|u\|_{L_2(\Pi)}^2=\int\limits_{\p\Pi} u(\D+1)U\di x=\int\limits_{\p\Pi} \left(u\frac{\p U}{\p \nu}-U\frac{\p u}{\p\nu}
\right)\di s,
\end{equation*}
where $\nu$ is the outward normal to $\p\Pi$. Thus,
\begin{equation}\label{4.7}
\begin{aligned}
2\|u\|_{L_2(\Pi)}^2=&\int\limits_{\G_\ell^+} u\frac{\p U}{\p x_2}\di x_1+\int\limits_{\g_\ell^+} U\frac{\p u}{\p x_2}\di x_1
\\
=&\int\limits_{\g_\ell^-} x_1 \frac{\p u}{\p x_2}\di x_1 \int\limits_{\G_\ell^+} u\di x_1 +\int\limits_{\g_\ell^+}   \frac{\p u}{\p x_2}\di x_1 \int\limits_{\G_\ell^-} x_1 u\di x_1.
\end{aligned}
\end{equation}
Consider two more integrals,
\begin{equation*}
0=\int\limits_{\Pi} x_1\sin \sqrt{E_1} x_2 (\D+E_1) u\di x,\quad
0=-\int\limits_{\Pi} \cos \sqrt{E_1} x_2 (\D+E_1)u\di x.
\end{equation*}
In each of them we integrate by parts using again the eigenvalue problem for $u$,
\begin{equation}\label{4.8}
\begin{aligned}
0=&\int\limits_{\p\Pi} \left(x_1\sin\sqrt{E_1}x_2\frac{\p u}{\p\nu}-u\frac{\p}{\p\nu}x_1\sin\sqrt{E_1}x_2 \right)\di x
\\
=&
\int\limits_{\g_\ell^-} x_1 \frac{\p u}{\p x_2}\di x_1 + \sqrt{E_1 } \int\limits_{\G_\ell^-} x_1 u\di x_1,
\\
0=&-\int\limits_{\p\Pi} \left(\cos\sqrt{E_1}x_2\frac{\p u}{\p\nu}-u\frac{\p}{\p\nu}\cos\sqrt{E_1}x_2 \right)\di x
\\
=
&
\int\limits_{\g_\ell^+} \frac{\p u}{\p x_2}\di x_1 - \sqrt{E_1} \int\limits_{\G_\ell^+} u\di x_1.
\end{aligned}
\end{equation}
Hence,
\begin{equation*}
\int\limits_{\g_\ell^-} x_1 \frac{\p u}{\p x_2}\di x_1=-\sqrt{E_1} \int\limits_{\G_\ell^-} x_1 u\di x_1,
\quad
\int\limits_{\g_\ell^+} \frac{\p u}{\p x_2}\di x_1=\sqrt{E_1} \int\limits_{\G_\ell^+} u\di x_1.
\end{equation*}
We substitute the obtained formulas in (\ref{4.7}) and obtain $2\|u\|_{L_2(\Pi)}^2=0$, which yields $u=0$. Thus, the problem
(\ref{v3.6}), (\ref{4.0}) has no exponentially decaying solutions. The proven fact and the symmetry of $\g_\ell$, $\G_\ell$, and $\Pi$ under the transformation (\ref{2.7a}) imply that the problem (\ref{v3.6}), (\ref{4.0}) can have at most two solutions with non-zero constants $c_\pm(0)$ such that one of these solutions is even under (\ref{2.7a}) while the other is odd. To complete the proof, it remains to show that the problem (\ref{v3.6}), (\ref{4.0}) can not have two such solutions simultaneously.

Assume such solutions exist. Without loss of generality we assume that for the even solution the coefficients are $c_\pm(0)=1$ while for the odd solution $c_\pm(0)=\pm 1$. Let $v$ be a sum of these solutions and introduce the function
\begin{equation*}
V(x):=\int\limits_{-\infty}^{x_1} v(t,x_2)\di t.
\end{equation*}
This function is well-defined since $v$ decays exponentially and satisfies  (\ref{4.0}) with $c_-(0)=0$, $c_+(0)=2$. Completely by analogy with the study of $U$
given above, we check that
\begin{align*}
&(\D+E_1) V=0\quad\text{in}\quad\Pi, \qquad
V=0\quad\text{on}\quad \g_\ell^-,\qquad\frac{\p V}{\p x_2}=0\quad\text{on}\quad\G_\ell^-,
\\
&
v=\int\limits_{\G_\ell^-} V\di x_1\quad\text{on}\quad \g_\ell^+,\qquad
\frac{\p v}{\p x_2}=\int\limits_{\g_\ell^-}\frac{\p V}{\p x_2}\di x_1\quad\text{on}\quad \g_\ell^-.
\end{align*}
We integrate by parts in the same way as in (\ref{4.7}),
\begin{equation}\label{4.9}
\begin{aligned}
0=&\int\limits_{\Pi} v (\D+E_1) V \di x=\int\limits_{\p\Pi} \left(v\frac{\p V}{\p\nu}-V\frac{\p v}{\p\nu}\right)\di s+4
\\
=&\int\limits_{\G_\ell^+} v\frac{\p V}{\p x_2}\di x_1 +\int\limits_{\g_\ell^+} V\frac{\p v}{\p x_1}\di x_1+4
\\
=&\int\limits_{\g_\ell^-} \frac{\p v}{\p x_2}\di x_1\int\limits_{\G_\ell^+} v \di x_1+ \int\limits_{\g_\ell^+} \frac{\p v}{\p x_2}\di x_1 \int\limits_{\G_\ell^-} v \di x_1+4.
\end{aligned}
\end{equation}
Proceeding as in (\ref{4.8}), we get
\begin{align*}
0=&\int\limits_{\p\Pi} \cos \sqrt{E_1}x_2 (\D+E_1) v\di x= \sqrt{E_1} \int\limits_{\G_\ell^+} v\di x_1 - \int\limits_{\g_\ell^+} \frac{\p v}{\p x_2}\di x_1,
\\
0=&\int\limits_{\Pi} \sin \sqrt{E_1} x_2 (\D+E_1)v\di x= \int\limits_{\g_\ell^-} \frac{\p v}{\p x_2}\di x_1+\sqrt{E_1} \int\limits_{\G_\ell^-} v\di x_1.
\end{align*}
These identities imply
\begin{equation*}
\int\limits_{\g_\ell^-} \frac{\p v}{\p x_2}\di x_1  \int\limits_{\G_\ell^+} v\di x_1 + \int\limits_{\g_\ell^+} \frac{\p v}{\p x_2}\di x_1 \int\limits_{\G_\ell^-} v\di x_1=0
\end{equation*}
that contradicts (\ref{4.9}).
\end{proof}

If for a given $\ell$ the nontrivial solution described in the previous lemma exists, we denote it by $\phi$.

We remind that the operator $\cc{T}_3(\mu)(\I+\cc{T}_4(\mu))^{-1}$ gives the solution to the problem (\ref{v3.6}), (\ref{4.0}). The next lemma describes the structure of this operator for small complex $\mu$.

\begin{lemma}\label{lm4.7}
Let $b>0$ be any number and $Q\in\Xi$.  If the solution $\phi$ does not exist, the operator
\begin{equation*}
\cc{T}_3(\I+\cc{T}_4)^{-1}: L_2(\Pi_a)\to \H^1(\Pi_b)\cap \H^2(Q)
\end{equation*}
is holomorphic w.r.t. small complex $\mu$. If the solution $\phi$ exists, the operator $\cc{T}_3(\I+\cc{T}_4)^{-1}$ is meromorphic w.r.t. small complex $\mu$, has a simple pole at $\mu=0$, and
\begin{equation}\label{v3.14}
\cc{T}_3(\mu)(\I+\cc{T}_4(\mu))^{-1}=\phi\frac{\cc{P}}{\mu}+\cc{T}_5(\mu),\quad \cc{P}:=\int\limits_{\Pi} \cdot\,\phi\di x,
\end{equation}
where the operator
\begin{equation*}
\cc{T}_5: L_2(\Pi_a)\to \H^1(\Pi_b)\cap \H^2(Q)
\end{equation*}
is bounded and holomorphic w.r.t. small complex $\mu$.
\end{lemma}
\begin{proof}
If the solution $\phi$ does not exist, the operator $(\I+\cc{T}_4(0))$ is boundedly invertible. Hence, by Lemma~\ref{vlm3.3} the inverse is holomorphic in small complex $\mu$. Now by Lemma~\ref{vlm3.4} the operator $\cc{T}_3(\mu)(\I+\cc{T}_4(\mu))^{-1}$ is well-defined, bounded, and holomorphic in $\mu$ as that from $L_2(\Pi_a)$ into $\H^1(\Pi_b)\cap \H^2(Q)$, where $b>0$, $Q\in\Xi$.

If $\phi$ exists, the proof of the representation (\ref{v3.14}) reproduces word-by-word the proof of Theorem~3.4 in \cite{BorExGad02}. Other statements of the lemma in this case follow from (\ref{v3.14}).
\end{proof}

In what follows for noncritical $l$ we let $\cc{P}:=0$ and under such definition the formula (\ref{v3.14}) is valid for all $\ell>0$.

\section{Emerging eigenvalues}\label{sec:emev}

In this section we study the eigenvalues close to the threshold of the essential spectrum. Namely, we prove items~\ref{it4th2.2},~\ref{it5th2.2} of Theorem~\ref{th2.2}. In the proof we follow the main lines of \cite{G02}, see also \cite[Sec. 4]{Bor-MSb06}, \cite[Sec. I\!I\!I.B]{BorExGad02}.

Let $\ell_*$ be a given value of $\ell$ and consider the value $\ell=\ell_*+\e$ with $\e$ small. For such values $\ell$ we consider the eigenvalue problem of $\Op_\ell$. We write it as the boundary value problem (\ref{v3.6}), (\ref{4.0}) with $f=0$ denoting the spectral parameter $\l=E_1-\mu^2$. In this boundary value problem we make a change of variables which rescales the sets $\g_\ell$ and $\G_\ell$ to $\g_{\ell_*}$ and $\G_{\ell_*}$. This change is introduced by (\ref{5.1}), where we assume that
the function $\xi_1$ is odd. It is easy to see that  under this change the problem (\ref{v3.6}) with $f=0$ casts into the form
\begin{align}
&
\begin{aligned}
&(-\D-E_1+\mu^2-\e\cc{L}_\e)\psi=0 \quad\text{in}\quad \Pi,
\\
&\psi=0\quad\text{on}\quad\g_{\ell_*},\qquad \frac{\p \psi}{\p y_2}=0\quad \text{on} \quad \G_{\ell_*},
\end{aligned}
\label{5.2}
\\
&\cc{L}_\e:=\Big(-2\xi_1'(x_1(y_1,\e))+\e\big( \xi_1'(x_1(y_1,\e))\big)^2
\Big) \frac{\p^2}{\p y_1^2}-\xi_1''(x_1(y_1,\e))\frac{\p}{\p y_1}, \label{5.3}
\end{align}
where we redenoted the eigenfunction by $\psi$.
This problem is completed by the condition (\ref{4.0}) for $\psi$, where $x$ should be replaced by $y$. By Lemma~\ref{vlm3.3} it is equivalent to the operator equation
\begin{equation}\label{5.4}
\psi=\e\cc{T}_3(\mu)\big(\I+\cc{T}_4(\mu)\big)^{-1} \cc{L}_\e \psi
\end{equation}
in $\H^1(\Pi_b)\cap\H^2(Q)$ for each $b>0$ and $Q\in\Xi$. The auxiliary parameter $a$ involved in the definition of $\cc{T}_3$ and $\cc{T}_4$ should be chosen large enough so that the supports of the coefficients of the operator $\cc{L}_\e$ lie inside $\Pi_a$.

By Lemma~\ref{vlm3.3}, if $\l=E_1-\mu^2$ is an eigenvalue of $\Op_\ell$, then the equation (\ref{5.4}) has a nontrivial solution and vice versa. This is why to prove items~\ref{it4th2.2},~\ref{it5th2.2} of Theorem~\ref{th2.2}
it is sufficient to study the existence of such solutions for sufficiently small $\e$.

We assume that the case when the solution $\phi$ described in Lemma~\ref{lm4.2} does not exist for $\ell=\ell_*$. In this case by Lemmas~\ref{vlm3.4},~\ref{lm4.7} the operator $\cc{T}_3(\mu)\big(\I+\cc{T}_4(\mu)\big)^{-1}\cc{L}_\e$ is bounded uniformly in small $\e$ and $\mu$. Hence, the equation (\ref{5.4}) has only trivial solution and we arrive at
\begin{lemma}\label{lm5.1}
Suppose for $\ell=\ell_*$ the solution $\phi$ does not exist. Then the operator $\Op_{\ell_*+\e}$ has no eigenvalues converging to $E_1$ as $\e\to+0$.
\end{lemma}

The rest of this section is devoted to the case of the existence of $\phi$. We suppose (\ref{2.19}) for this function.
We employ the representation (\ref{v3.14}) to rewrite the equation (\ref{5.4}) as
\begin{equation}\label{5.5}
\psi=\frac{\e\cc{P}\cc{L}_\e \psi}{\mu}\phi +\e \cc{T}_5(\mu)\cc{L}_\e \psi.
\end{equation}
By Lemma~\ref{lm4.7} and the definition of $\cc{L}_\e$ the operator $\cc{T}_5(\mu)\cc{L}_\e$ is bounded in $\H^1(\Pi_b)\cap \H^2(Q)$ and jointly holomorphic in $\mu$ and $\e$. Hence, the operator $\big(\I-\e \cc{T}_5(\mu)\cc{L}_\e\big)$ is invertible and we apply the inverse to (\ref{5.5}),
\begin{equation}\label{5.6}
\psi=\frac{\e\cc{P}\cc{L}_\e \psi}{\mu} \big(\I-\e \cc{T}_5(\mu)\cc{L}_\e\big)^{-1} \phi.
\end{equation}
Since $\psi$ is an eigenfunction, it is non-zero and in view of the last equation it yields $\cc{P}\cc{L}_\e \psi\not=0$. Bearing in mind this inequality, we apply the functional $\cc{P}\cc{L}_\e$ to (\ref{5.6}) and divide it then by $\cc{P}\cc{L}_\e \psi$,
\begin{align}
&1=\frac{\e}{\mu}  \cc{P}\cc{L}_\e \big(\I-\e \cc{T}_5(\mu)\cc{L}_\e\big)^{-1} \phi,\nonumber
\\
&\mu=\e \cc{P} \big(\I-\e \cc{L}_\e\cc{T}_5(\mu)\big)^{-1} \cc{L}_\e\phi.
\label{5.7}
\end{align}
This is the equation determining the values of $\mu$ for which the equation (\ref{5.4}) has a nontrivial solution. This solution is given by (\ref{5.6}). Indeed, the fraction $\frac{\e\cc{P}\cc{L}_\e \psi}{\mu}$ is a constant, and we can determine $\psi$ as
\begin{equation}\label{5.8}
\psi=C \big(\I-\e \cc{T}_5(\mu)\cc{L}_\e\big)^{-1} \phi,\quad C=\mathrm{const}.
\end{equation}
In what follows we shall show that the most convenient way to choose $C$ is to let $C=1$.

By Lemma~\ref{lm4.7} the function
\begin{equation*}
F(\mu,\e):=\cc{P} \big(\I-\e \cc{L}_\e\cc{T}_5(\mu)\big)^{-1} \cc{L}_\e\phi
\end{equation*}
is jointly holomorphic in small $\mu$ and $\e$. The derivative
$\frac{\p F}{\p\mu}(0,\e)$ is bounded uniformly in $\e$. Hence, by the implicit function theorem the equation (\ref{5.7}) has precisely one root $\mu_\e$, which is holomorphic in $\e$. It is exactly the value for which the problem (\ref{5.2}), (\ref{5.3}) has a nontrivial solution.  In view of (\ref{4.0}) this solution is an eigenfunction if $\mu_\e>0$. In this case the eigenvalue itself is given by the formula
\begin{equation}\label{4.32}
\l_\e=E_1-\mu_\e^2.
\end{equation}
In what follows we shall show that indeed $\mu_\e>0$.

Since the root $\mu_\e$ to the equation  (\ref{5.7}) is holomorphic in $\e$, it can be represented as a convergent series
\begin{equation}\label{5.9}
\mu_\e=\sum\limits_{j=1}^{\infty} \e^j \mu_j.
\end{equation}
The equation (\ref{5.8}) implies that the associated nontrivial solution $\psi_\e$ to (\ref{5.2}), (\ref{4.0}) can be chosen holomorphic in  $\e$ in $\H^1(\Pi_b)\cap\H^2(Q)$,
\begin{equation}\label{5.11}
\psi_\e(y)=\sum\limits_{j=0}^{\infty} \e^j \psi_j(y),
\end{equation}
Let us determine $\mu_j$ and $\psi_j$.

Consider the eigenvalue problem (\ref{5.2}), (\ref{4.0}) for $\psi_\e$. The coefficients of the operator $\cc{L}_\e$ are compactly supported with supports lying inside $\Pi_{2\ell_*}$. Outside $\Pi_{2\ell_*}$ we can solve the problem (\ref{5.2}) by the separation of variables and in view of the asymptotics (\ref{4.0}) it implies
\begin{equation*}
\psi_\e(x)=\sum\limits_{m=1}^{\infty} A_m^\pm(\e) \chi_m(x_2) \E^{-\sqrt{E_m-E_1+\mu_\e^2}|x_1|},\quad |x_1|>2\ell_*.
\end{equation*}
It is easy to see that this identity is equivalent to the following ones,
\begin{equation}\label{5.34}
\left(\frac{\p}{\p y_1}\pm\sqrt{E_m-E_1+\mu_\e^2}\right) \int\limits_0^d \psi_\e(y)\chi_m(y)\di y_2=0,\quad y_1=\pm 2\ell_*.
\end{equation}
To determine $\psi_j$ we shall deal with the problem (\ref{5.2}), (\ref{5.34}) instead of (\ref{5.2}), (\ref{4.0}).

We substitute the series (\ref{5.9}), (\ref{5.11}) into (\ref{5.2}) and collect the coefficients at like powers of $\e$. It yields the equations for $\psi_j$,
\begin{equation}\label{5.35}
\begin{aligned}
-&(\D+E_1)\psi_j=-\sum\limits_{p=2}^{j} M_p \psi_{j-p} +\sum\limits_{p=1}^{j} \cc{L}_p\psi_{j-p}\quad\text{in}\quad\Pi_{2\ell_*},
\\
&\psi_j=0\quad\text{on}\quad \g_{\ell_*}\cap \p\Pi_{2\ell_*},\qquad \frac{\p\psi_j}{\p y_2}=0 \quad\text{on}\quad \G_{\ell_*}\cap\p\Pi_{2\ell_*},
\end{aligned}
\end{equation}
where $M_p$ are the coefficients of the square of $\mu_\e^2$, i.e.,
\begin{equation}\label{5.36}
M_p:=\sum\limits_{k=1}^{p-1}\mu_k \mu_{p-k},
\end{equation}
and $\cc{L}_p$ are differential operators with infinitely differentiable compactly supported coefficients,
\begin{equation*}
\cc{L}_p=K_{11}^{(p)}(y_1)\frac{\p^2}{\p y_1^2}+ K_1^{(p)}(y_1) \frac{\p}{\p y_1}.
\end{equation*}
In particular,
\begin{align}
& \cc{L}_1=-2\xi_1'(y_1)\frac{\p^2}{\p y_1^2}-\xi_1''(y_1) \frac{\p}{\p y_1},\label{5.38}
\\
&\cc{L}_2= \big( {\xi_1'}^2(y_1)-2\xi_1''(y_1)\xi_1(y_1)
\big) \frac{\p^2}{\p y_1^2} - \xi_1''' (y_1) \xi_1(y_1) \frac{\p}{\p y_1}.\label{4.27}
\end{align}
In the same way we rewrite the conditions (\ref{5.34}) for the functions $\psi_j$,
\begin{equation}
\frac{\p}{\p y_1} \int\limits_0^d \psi_j(y)\chi_1(y)\di y_2= \mp \sum\limits_{p=1}^{j} \mu_p \int\limits_{0}^{d} \psi_{j-p}(y)\chi_1(y)\di y_2
\label{5.40}
\end{equation}
as $y_1=\pm 2\ell_*$, and
\begin{equation}\label{5.41}
\begin{aligned}
& \left(\frac{\p}{\p y_1}\pm \sqrt{E_m-E_1}\right) \int\limits_{0}^{d} \psi_j(y) \chi_m(y)\di y_2
\\
&\hphantom{\Bigg(}= \mp \sqrt{E_m-E_1} \sum\limits_{p=2}^{j} b_p \left(\frac{\mu_1}{\sqrt{E_m-E_1}},\ldots,\frac{\mu_{p-1}}{\sqrt{E_m-E_1}}\right)
\int\limits_{0}^{d} \psi_{j-p}(y)\chi_m(y)\di y_2
\end{aligned}
\end{equation}
as $y_1=\pm 2\ell_*$, $\quad m\geqslant 2$, where $b_p=b_p(z_1,\ldots,z_{p-1})$ are certain polynomials in $(z_1,\ldots,z_{p-1})$, and, in particular,
\begin{equation}\label{5.42}
b_2(z_1)=\frac{z_1^2}{2}.
\end{equation}
To solve the obtained problems for $\psi_j$, we shall make use of an auxiliary
\begin{lemma}\label{lm5.3}
Given a function $f\in L_2(\Pi_{2\ell_*})$ and two sequences $\{a_m^\pm\}_{m=1}^\infty$ such that
\begin{equation}\label{5.43}
\sum\limits_{m=1}^{\infty} |a_m^\pm|^2<\infty,
\end{equation}
consider the boundary value problem
\begin{align}
&
\begin{gathered}
-(\D+E_1)u=f\quad \text{in}\quad \Pi_{2\ell_*},
\\
u=0\quad\text{on}\quad \g_{\ell_*}\cap \p\Pi_{2\ell_*},\qquad
\frac{\p u}{\p y_2}=0\quad \text{on}\quad \G_{\ell_*} \cap \p\Pi_{2\ell_*},
\end{gathered}\label{5.44}
\\
&\left(\frac{\p}{\p y_1} \pm \sqrt{E_m-E_1}
\right) \int\limits_{0}^{d}\psi_j(y)
\chi_m(y)\di y_2=a_m^\pm.\label{5.45}
\end{align}
This problem is solvable in $\H^1(\Pi_{2\ell_*})$, if and only if the solvability condition
\begin{equation}\label{5.46}
\int\limits_{\Pi_{2\ell_*}} f\phi \di x= \sum\limits_{m=1}^{\infty} (\wp a_m^- - a_m^+) a_m^*
\end{equation}
holds true, where $\wp:=1$, if $\phi$ is even under the transformation (\ref{2.7a}), and  $\wp:=-1$, if $\phi$ is odd. The coefficients $a_m^*$ are determined by the equation
\begin{equation}\label{5.47}
\phi(\ell_*,y_2)=\sum\limits_{m=1}^{\infty} a_m^* \chi_m(\ell_*,y_2),\quad a_m^*=\int\limits_{0}^{d} \phi(\ell_*,y_2) \chi_m(\ell_*,y_2)\di y_2.
\end{equation}
The solution to the problem (\ref{5.44}), (\ref{5.45})
is determined up to an additive term $C\phi$. There exists the unique solution to  (\ref{5.44}), (\ref{5.45})
satisfying the condition
\begin{equation}\label{4.22}
\int\limits_{0}^{d} u(2\ell_*,y_2)\chi_1(2\ell_*,y_2)\di y_2+\wp \int\limits_{0}^{d} u(-2\ell_*,y_2)\chi_1(-2\ell_*,y_2)\di y_2=0.
\end{equation}
\end{lemma}

\begin{proof}
We first construct a function satisfying (\ref{5.45}),
\begin{equation*}
u_0(y):=\pm\xi_2(y_1) \sum\limits_{m=1}^{\infty} \frac{a_m^\pm}{2\sqrt{E_m-E_1}} \chi_m(y) \E^{\pm\sqrt{E_m-E_1} (y_1\mp 2\ell_*)},\quad \pm y_1>0,
\end{equation*}
where the function $\xi_2$ introduced in (\ref{v3.10}) is taken for $\ell=\ell_*$, $a=2\ell_*$. Due to the presence of the exponent in the definition of $u_0$ and (\ref{5.43}) this function is well-defined in $\Pi_{2\ell_*}$. Namely, it is infinitely differentiable in $\Pi_{2\ell_*}$, belongs to $\H^1(\Pi_{2\ell_*})$ and satisfies boundary conditions in (\ref{5.44}) and the equation
\begin{equation}\label{4.21}
\begin{aligned}
-(\D+E_1) u_0(y)=:f_0(y)=&\mp\xi_2''(y_1) \sum\limits_{m=1}^{\infty}  \frac{\sqrt{E_m-E_1}}{2} \chi_m(y) \E^{\pm \sqrt{E_m-E_1}(y_1\mp 2\ell_*)}
\\
&-\xi_2'(y_1) \sum\limits_{m=1}^{\infty} \chi_m(y) \E^{\pm \sqrt{E_m-E_1}(y_1\mp 2\ell_*)},\quad \pm y_1>0.
\end{aligned}
\end{equation}

We construct the solution to (\ref{5.44}), (\ref{5.45}) as $u=u_0+u_1$. In view of the aforementioned properties of $u_0$ the function $u_1$ should solve the problem (\ref{5.44}), (\ref{5.45}), but with $f$ replaced by $(f-f_0)$ and $a_m^\pm$ replaced by zero. We extend the function $u_1$ as
\begin{equation*}
u_1(y):=\sum\limits_{m=1}^{\infty} \chi_m(y) \E^{\mp \sqrt{E_m-E_1}(y_1\mp 2\ell_*)} \int\limits_{0}^{d} u_1(\pm 2\ell_*,t) \chi_m(\pm 2\ell_*,t)\di t,\quad \pm y_1>2\ell_*
\end{equation*}
and see that then it is a bounded solution to the problem (\ref{v3.6}), (\ref{4.0}) with the compactly supported right hand side $(f-f_0)$ and $\mu=0$. By Lemma~\ref{lm4.7} such solution exists, if and only if the right hand side of the equation satisfies the condition,
\begin{align*}
&\int\limits_{\Pi} (f-f_0)\phi\di y=\int\limits_{\Pi_{2\ell_*}} (f-f_0)\phi\di y=0,
\\
&\int\limits_{\Pi_{2\ell_*}} f\phi\di y =\int\limits_{\Pi_{2\ell_*}} f_0\phi \di y.
\end{align*}
We substitute (\ref{4.21}) into the right hand side of the last identity and integrate by parts,
\begin{align*}
\int\limits_{\Pi_{2\ell_*}} f_0\phi\di y=&- \int\limits_{\Pi_{2\ell_*}}  \phi (\D+E_1)u_0\di y
\\
=&
\int\limits_{x_1=-2\ell_*} \left(\phi \frac{\p u_0}{\p y_1}- u_0 \frac{\p\phi}{\p y_1} \right)\di y_2-\int\limits_{x_1=2\ell_*} \left(\phi \frac{\p u_0}{\p y_1}- u_0 \frac{\p\phi}{\p y_1} \right)\di y_2
\\
=&\sum\limits_{m=1}^{\infty} (\wp a_m^- - a_m^+) a_m^*.
\end{align*}
It proves (\ref{5.46}).

As $f=0$, the problem (\ref{5.44}), (\ref{5.45}) has the only nontrivial solution, which is $\phi$. Hence, the generalized solution to (\ref{5.44}), (\ref{5.45}) is defined up to an additive term $C\phi$. It is also clear that there exists the unique solution satisfying (\ref{4.22}).
\end{proof}

We apply the proven lemma to study the solvability of the problems (\ref{5.35}), (\ref{5.40}), (\ref{5.41}). It follows from (\ref{5.8}) that
\begin{equation}\label{4.23}
\psi_0=\phi.
\end{equation}
It also fits the problem (\ref{5.35}), (\ref{5.40}), (\ref{5.41}) for $\psi_0$ and the definition of $\phi$.

Consider the problem for $\psi_1$. Due to (\ref{4.23}), (\ref{5.36}), (\ref{5.38}) it casts into the form
\begin{equation}\label{4.24}
\begin{aligned}
&-(\D+E_1)\psi_1= \cc{L}_1\phi\quad\text{in}\quad\Pi_{2\ell_*},
\\
&\psi_1=0\quad\text{on}\quad \g_{\ell_*}\cap \p\Pi_{2\ell_*},\qquad \frac{\p\psi_1}{\p y_2}=0 \quad\text{on}\quad \G_{\ell_*}\cap\p\Pi_{2\ell_*},
\\
&
\frac{\p}{\p y_1} \int\limits_0^d \psi_1(y)\chi_1(y)\di y_2= \mp \mu_1\int\limits_{0}^{d} \phi(y)\chi_1(y)\di y_2, \quad y_1=\pm 2\ell_*,
\\
& \left(\frac{\p}{\p y_1}\pm \sqrt{E_m-E_1}\right) \int\limits_{0}^{d} \psi_1(y) \chi_m(y)\di y_2=0,\quad y_1=\pm 2\ell_*, \quad m\geqslant 2.
\end{aligned}
\end{equation}
In accordance with Lemma~\ref{lm5.3} and (\ref{5.47}) the solvability condition for this problem is
\begin{equation*}
\int\limits_{\Pi_{2\ell_*}} \phi \cc{L}_1\phi \di y=2\mu_1 |a_1^*|^2.
\end{equation*}
Employing (\ref{5.38}), (\ref{2.19}) and the equation for $\phi$ in (\ref{2.18}), we obtain
\begin{align}
&\cc{L}_1\phi=-(\D+E_1) \xi_1 \frac{\p\phi}{\p y_1},\label{4.34}
\\
&\mu_1=\frac{1}{2} \int\limits_{\Pi_{2\ell_*}} \phi \cc{L}_1\phi \di y=- \frac{1}{2} \int\limits_{\Pi_{2\ell_*}} \phi (\D+E_1) \xi_1 \frac{\p \phi}{\p y_1}\di y. \nonumber
\end{align}
By analogy with  Lemma~4.2 in \cite{Bor-MSb06} one can show that at $(\ell_*,0)$ the function $\phi$ has a differentiable asymptotics
\begin{equation}\label{4.29}
\phi(y)=\a_1 r^{\frac{1}{2}}\sin\frac{\tht}{2} +\a_3 r^{\frac{3}{2}}\sin\frac{3\tht}{2}+ \Odr(r^{\frac{5}{2}}),\quad r\to+0, \quad \a_1, \a_2=\mathrm{const},
\end{equation}
where $(r,\tht)$ are polar coordinates centered at $(\ell_*,0)$. The similar behavior holds true at $(-\ell_*,d)$ but with the coefficients replaced by $\wp\a_1$ and $\wp\a_2$ due to the  parity of $\phi$ under (\ref{2.7a}).

The coefficient $\a_1$ satisfies the identity
\begin{equation*}
   \a_1^2= \frac{2}{\pi\ell_*} \int\limits_{\Pi} \left|\frac{\p \phi_n}{\p x_1}\right|^2\di x,
\end{equation*}
which we prove by the integration by parts employing (\ref{4.29}),
\begin{align*}
0=&\int\limits_{\Pi}  y_1 \phi (\D+E_1)\frac{\p\phi}{\p y_1}\di y=-\frac{\pi \ell_* \a_1^2}{2}+ 2\int\limits_{\Pi} \left|\frac{\p\phi}{\p y_1}\right|^2\di y.
\end{align*}
In the same way we calculate $\mu_1$,
\begin{equation}\label{4.25}
\mu_1=-\frac{1}{2} \int\limits_{\Pi_{2\ell_*}} \phi (\D+E_1) \xi_1\frac{\p \phi}{\p y_1}\di y
=\frac{\pi \a_1^2}{4}= \frac{1}{\ell_*} \int\limits_{\Pi} \left|\frac{\p \phi}{\p y_1}\right|^2\di y.
\end{equation}
It proves (\ref{2.24}). The proven formula for $\mu_1$ implies that $\mu_1>0$. Hence,  by (\ref{5.9}), the root $\mu_\e$ is positive and the associated the nontrivial solution to the problem (\ref{5.2}), (\ref{4.0}) decays exponentially at infinity and thus belongs to $\Ho^1(\Pi,\g_{\ell_*})$. Therefore, up to the change (\ref{5.1}) it is an eigenfunction of the operator  $\Op_{\ell_*+\e}$ for sufficiently small $\e$ associated with the eigenvalue given by (\ref{4.32})

The function $\psi_1$ is defined up to an additive term $C\phi$. We fix it uniquely assuming (\ref{4.22}) for $\psi_1$. Since the function $\xi_1$ is odd and the function $\psi_1$ is defined uniquely, it follows from the problem (4.24) that the function $\psi_1$ has the same parity under the transformation (\ref{2.7a}) as $\phi$ does. By analogy with Lemma~4.2 in \cite{Bor-MSb06} one can also show that at $(\ell_*,0)$ the function $\psi_1$ has the same behavior at $(\ell_*,0)$ and $(-\ell_*,d)$ as $\phi$. Namely,
\begin{equation}\label{4.31}
\psi_1(y)=\a_* r^{\frac{1}{2}}\sin\frac{\tht}{2}  + \Odr(r^{\frac{3}{2}}),\quad r\to+0, \quad \a_*=\mathrm{const},
\end{equation}
Although the problem (\ref{4.24}) involves the function $\xi_1$, the constant $\a_1^{(1)}$ is independent on the choice of $\xi_1$. Indeed, if we choose another cut-off function $\widetilde{\xi}_1$, it it easy to see that the corresponding solution $\widetilde{\psi}_1$ will be
\begin{equation*}
\widetilde{\psi}_1=\psi_1+(\widetilde{\xi}_1-\xi_1)\frac{\p\phi}{\p y_1}.
\end{equation*}
Since the function $(\widetilde{\xi}_1-\xi_1)$ vanishes in a neighborhood of $y_1=\ell_*$, the term $(\widetilde{\xi}_2-\xi_1)\frac{\p\phi}{\p y_1}$ can not change the constant $\a_*$ in (\ref{4.31}).

The problems for $\psi_j$, $j\geqslant 2$, are studied in the same way. Each of these functions is assumed to satisfy (\ref{4.22}). The condition (\ref{5.43}) for the right hand sides in (\ref{5.41}) follow from the belongings $\psi_j(\pm 2\ell_*,\cdot)\in L_2(0,d)$ and obvious estimates
\begin{equation*}
\left|\sqrt{E_m-E_1} b_p
\left(\frac{\mu_1}{\sqrt{E_m-E_1}},\ldots,\frac{\mu_{p-1}}{\sqrt{E_m-E_1}}\right)
\right|\leqslant C_p,
\end{equation*}
where the constants $C_p$ are independent of $m$. The solvability condition (\ref{5.46}) together with the assumptions (\ref{4.22}) determine the coefficients $\mu_j$,
\begin{align}
&
\begin{aligned}
2\mu_j&+\sum\limits_{m=2}^{\infty} a_m^*\sqrt{E_m-E_1} \sum\limits_{p=2}^{j} b_p
\left(\frac{\mu_1}{\sqrt{E_m-E_1}},\ldots, \frac{\mu_{p-1}}{\sqrt{E_m-E_1}}\right)
\\
& \int\limits_{0}^{d}\Big(\wp \psi_{j-p}(-2\ell_*,y_2)\chi_m(-2\ell_*,y_2) + \psi_{j-p}(2\ell_*,y_2) \chi_m(2\ell_*,y_2)
\Big)\di y
\\
=&- \sum\limits_{p=2}^{j} M_p \int\limits_{\Pi_{2\ell_*}} \psi_{j-p} \phi \di y+ \sum\limits_{p=1}^{j} \int\limits_{\Pi_{2\ell_*}} \phi \cc{L}_p \psi_{j-p} \di y,
\end{aligned}\nonumber
\\
&
\begin{aligned}
\mu_j=&-\frac{1}{2} \sum\limits_{p=2}^{j} M_p \int\limits_{\Pi_{2\ell_*}} \psi_{j-p} \phi \di y
+ \frac{1}{2}
\sum\limits_{p=1}^{j} \int\limits_{\Pi_{2\ell_*}} \phi \cc{L}_p \psi_{j-p} \di y
\\
&
- \frac{1}{2} \sum\limits_{m=2}^{\infty} a_m^*\sqrt{E_m-E_1} \sum\limits_{p=2}^{j} b_p
\left(\frac{\mu_1}{\sqrt{E_m-E_1}},\ldots, \frac{\mu_{p-1}}{\sqrt{E_m-E_1}}\right)
\\
& \int\limits_{0}^{d}\Big(\wp \psi_{j-p}(-2\ell_*,y_2)\chi_m(-2\ell_*,y_2) + \psi_{j-p}(2\ell_*,y_2) \chi_m(2\ell_*,y_2)
\Big)\di y.
\end{aligned}\label{4.26}
\end{align}

The functions $\psi_n$ and $\phi_n$ have certain parity under the symmetry transformation (\ref{2.7a}) and $\psi_n$ converges to $\phi_n$ as $\e\to+0$ in $\H^1(\Pi_b)$ for each $b>0$. Hence, the function $\phi_n$ has the same parity as $\psi_n$. In view of item~\ref{it3th2.2} of Theorem~\ref{th2.2} it proves the required statement on the parity of $\phi_n$.

Let us calculate $\mu_2$. We substitute (\ref{5.36}), (\ref{5.38}), (\ref{4.27}), (\ref{5.42}) into (\ref{4.26}) with $j=2$,
\begin{equation}
\mu_2=
 -\frac{\mu_1^2}{2} \int\limits_{\Pi_{2\ell_*}}  |\phi|^2\di y + \frac{1}{2} \int\limits_{\Pi_{2\ell_*}}
\phi (\cc{L}_1\psi_1 + \cc{L}_2\phi)\di y - \mu_1^2 \sum\limits_{m=2}^{\infty} \frac{|a_m^*|^2}{\sqrt{E_m-E_1}}
\label{4.28}
\end{equation}

Let us  construct the function $\psi_1$ via $\xi_1$, $\phi$ and the solution to the problem (\ref{2.30}). It follows from (\ref{4.29}) that
\begin{equation}\label{4.33}
\frac{\p\phi}{\p y_1}(y)=-\frac{1}{2} \a_1 r^{-\frac{1}{2}} \sin \frac{\tht}{2}  + \frac{3}{2} \a_3 r^{\frac{1}{2}} \sin \frac{\tht}{2} + \Odr(r^{\frac{3}{2}}), \quad x\to(\ell_*,0).
\end{equation}
In view of this asymptotics the equations (\ref{4.24}), (\ref{4.34}) yield that the function
\begin{equation}\label{4.35}
\ph=\psi_1-\xi_1\frac{\p\phi}{\p y_1}-\mu_1\ell_*\phi
\end{equation}
is exactly the solution to the problem (\ref{2.30}) with $\ell_n=\ell_*$, $\phi_n=\phi$, $\mu_1^{(n)}=\mu_1$. The solution to this problem is unique, since by Lemma~\ref{lm4.2} the corresponding homogenous problem has the trivial solution only. It follows from (\ref{4.35}) that
\begin{equation}\label{4.36}
\psi_1=\ph+\xi_1\frac{\p\phi}{\p y_1}+\mu_1\ell_*\phi.
\end{equation}
By analogy with Lemma~4.2 in \cite{Bor-MSb06} one can show that
\begin{equation}\label{4.38}
\ph(y)=\frac{\mu_1^{\frac{1}{2}}}{\pi^{\frac{1}{2}}} r^{-\frac{1}{2}} \sin \frac{\tht}{2} +\b r^{\frac{1}{2}}\sin \frac{\tht}{2} + \Odr(r^{\frac{3}{2}}),\quad x\to(\ell_*,0),\quad \b=\mathrm{const}.
\end{equation}
Together with (\ref{4.29}), (\ref{4.25}), (\ref{4.31}), (\ref{4.36}) it yields
\begin{equation}\label{4.37}
\begin{aligned}
&\psi_1(y)=\widehat{\b} r^{\frac{1}{2}} \sin \frac{\tht}{2} + \Odr(r^{\frac{1}{2}}),\quad x\to(\ell_*,0), \quad \widehat{\b}=\b + \frac{3\a_3}{2} + \frac{\pi \a_1^3\ell_*}{4}.
 \end{aligned}
\end{equation}
Employing  (\ref{5.38}), (\ref{4.27}), the equation for $\psi_1$ in (\ref{4.24}) and that for $\phi$ in (\ref{2.18}), by direct calculations we check that
\begin{align*}
&\frac{1}{2} \cc{L}_1\psi_1= -\frac{1}{2} (\D+E_1) \xi_1 \frac{\p\psi_1}{\p y_1} + \frac{1}{2} \xi_1 \frac{\p}{\p y_1} (\D+E_1)\psi_1
\\
& \hphantom{\frac{1}{2} \cc{L}_1\psi_1}=-\frac{1}{2} (\D+E_1) \xi_1 \frac{\p \psi_1}{\p y_1} + \frac{1}{2} \xi_1 \frac{\p}{\p y_1} (\D+E_1) \xi_1 \frac{\p \phi}{\p y_1},
\\
&\frac{1}{2} \cc{L}_2\phi= -\frac{1}{4} (\D+E_1) \frac{\p}{\p y_1} \xi_1^2 \frac{\p\phi}{\p y_1} - \frac{1}{2} \cc{L}_1 \xi_1 \frac{\p\phi}{\p y_1}
\\
&\hphantom{\frac{1}{2} \cc{L}_2\phi}= -\frac{1}{4} (\D+E_1) \frac{\p}{\p y_1} \xi_1^2 \frac{\p\phi}{\p y_1} + \frac{1}{2} (\D+E_1) \xi_1 \frac{\p}{\p y_1}\xi_1 \frac{\p\phi}{\p y_1} - \frac{1}{2} \xi_1 (\D+E_1) \frac{\p}{\p y_1}\xi_1 \frac{\p\phi}{\p y_1},
\\
&\frac{1}{2} (\cc{L}_1\psi_1+\cc{L}_2 \phi) = - \frac{1}{2} (\D+E_1) \xi_1 \frac{\p\psi_1}{\p y_1}
 -\frac{1}{4} (\D+E_1) \frac{\p}{\p y_1} \xi_1^2 \frac{\p\phi}{\p y_1} + \frac{1}{2}(\D+E_1) \xi_1 \frac{\p}{\p y_1} \xi_1 \frac{\p\phi}{\p y_1}
 \\
&\hphantom{\frac{1}{2} (\cc{L}_1\psi_1+\cc{L}_2 \phi)} = -\frac{1}{2}(\D+E_1) \xi_1 \frac{\p\psi_1}{\p y_1} + \frac{1}{4} (\D+E_1) \xi_1^2 \frac{\p^2\phi}{\p y_1^2}.
\end{align*}
We substitute the last expression into the second term in the right hand side of (\ref{4.28}) and integrate by parts taking into consideration (\ref{4.29}), (\ref{4.37}),
\begin{equation*}
\frac{1}{2} \int\limits_{\Pi_{2\ell_*}}
\phi (\cc{L}_1\psi_1 + \cc{L}_2\phi)\di y=\int\limits_{\Pi}
\phi \left(   \frac{1}{4} (\D+E_1) \xi_1^2 \frac{\p^2\phi}{\p y_1^2}-\frac{1}{2}(\D+E_1) \xi_1 \frac{\p\psi_1}{\p y_1}\right)\di y=\frac{\pi \a_1 \widehat{\b}}{4}.
\end{equation*}
We integrate by parts once again bearing in mind (\ref{4.25}), (\ref{4.31}), (\ref{4.37}),
\begin{align*}
&0= \int\limits_{\Pi} y_1\phi (\D+E_1)  \frac{\p\ph}{\p y_1}\di y= -\frac{3\pi \a_1^2}{8} - \frac{\pi \a_1\ell_*}{2} \left(\b+ \frac{3\a_3}{2}\right)+ 2\int\limits_{\Pi} \frac{\p\phi}{\p y_1} \frac{\p\ph}{\p y_1}\di y,
\\
&\frac{\pi \a_1}{4}\left(\b+ \frac{3\a_3}{2}\right)=-\frac{3\pi \a_1^2}{16}  + \frac{1}{\ell_*}\int\limits_{\Pi} \frac{\p\phi}{\p y_1} \frac{\p\ph}{\p y_1}\di y.
\end{align*}
The integral in the right hand side of the last identity converges due to (\ref{4.33}), (\ref{4.38}).
The last obtained identity and (\ref{4.31}), (\ref{4.33}), (\ref{4.37}) yield the formula (\ref{2.29}) for $\mu_2$.

\section*{Acknowledgments}

D.B. was partially supported by RFBR, the grant of the President of Russia for leading scientific school, by the Federal Task Program ``Scientific and pedagogical staff of innovative Russia for 2009-2013'' (contract no. 02.740.11.0612) and by the grant of FCT (ptdc/mat/101007/2008)

\end{document}